\documentclass[10pt]{article}
\usepackage{amsmath,amsthm,amssymb,eucal,tocloft} 
\usepackage{pdfsync} 
\usepackage{cite}
\usepackage{setspace}
\usepackage[all,cmtip]{xy}
\usepackage{graphicx}
\usepackage{subfig}

\usepackage{mycolor}
\usepackage[pagebackref=true]{hyperref}
\hypersetup{
	colorlinks,
    citecolor=mydarkblue,
    filecolor=mydarkblue,
    linkcolor=mydarkblue,
    urlcolor=mydarkblue
}

\newtheorem{thm}{Theorem}
\newtheorem{cor}[thm]{Corollary} 
\newtheorem{lem}[thm]{Lemma} 

\theoremstyle{definition} 
\newtheorem{defn}[thm]{Definition}
\theoremstyle{definition} 

\theoremstyle{definition} 
\newtheorem{remark}[thm]{Remark}
\theoremstyle{definition} 

\theoremstyle{definition} 

\theoremstyle{definition} 
\newtheorem{example}[thm]{Example}

\numberwithin{thm}{subsection}

\newcommand{\R}{\ensuremath{\mathbb{R}}} 
\newcommand{\N}{\ensuremath{\mathbb{N}}} 
\newcommand{\C}{\ensuremath{\mathbb{C}}} 
 
\newcommand{\F}{\ensuremath{\mathbb{F}}}

\newcommand{\<}{\langle} 
\renewcommand{\>}{\rangle} 
 
\def\p{\partial} 
\def\k{\kappa} 
\def\i{\infty}
\def\a{\alpha}
\def\e{\epsilon}

\def\supp{\it{supp}}

\def\b{\beta} 
 
\def\d{\delta}  
\def\l{\lambda} 
\def\L{\Lambda} 
\def\o{\omega} 
\def\O{\Omega}

\def\D{\Delta}    
\def\G{\Gamma}
\def\B{\mathcal{B}}
\def\hB{\hat{\mathcal{B}}}
\def\P{\mathcal{P}}
\def\lip{\mathrm{Lip}}

\def\Xint#1{\mathchoice
{\XXint\displaystyle\textstyle{#1}}%
{\XXint\textstyle\scriptstyle{#1}}%
{\XXint\scriptstyle\scriptscriptstyle{#1}}%
{\XXint\scriptscriptstyle\scriptscriptstyle{#1}}%
\!\int}
\def\XXint#1#2#3{{\setbox0=\hbox{$#1{#2#3}{\int}$}
\vcenter{\hbox{$#2#3$}}\kern-.5\wd0}}

\def\cint{\Xint \smallsetminus}

\begin{document}

	\begin{center}
\thispagestyle{empty}
	\textsc{}\\[3cm]
	\huge Applications of Differential Chains to Complex Analysis and Dynamics\\[1.5cm]

	\begin{minipage}{0.4\textwidth}
	\begin{flushleft} \large
	\emph{Author:}\\
	Harrison Pugh
	\end{flushleft}
	\end{minipage}
	\begin{minipage}{0.4\textwidth}
	\begin{flushright} \large
	\emph{Advisor:} \\
	Sarah Koch
	\end{flushright}
	\end{minipage}
	\vfill
	{\large November 30, 2009}
	\end{center}

\newpage

	\addcontentsline{toc}{section}{Acknowledgements}
				\begin{center}
			\Large{Acknowledgements}
		\end{center}
		
			First and foremost, I would like to thank my advisor Sarah Koch for being so generous with her time and energy.  She tirelessly read draft after draft, and always had helpful suggestions on how to make my exposition clearer.  Without her, this undertaking would have fallen flat (or perhaps sharp, as it were.)  I am deeply grateful for her help.  I would also like to thank my mother Jenny Harrison, on whose work this thesis is partially based.  Her patience in explaining mathematics to me has been unwavering over the years, especially during these recent months during which my questions have been ceaseless.  I cannot begin to thank her enough for encouraging me to be creative and for inspiring me to pursue a career in mathematics.  Finally, I would like to thank my mentor Moe Hirsch, who generously opened his home to me this summer and showed me a beautiful world of math and music.  I learned a great deal from our conversations, and I am honored and humbled to have studied under him.

	\newpage

	\tableofcontents
	
	\newpage

\section{Introduction}
\subsection{Overview}
This thesis is divided into three parts.  In the first part, we give an introduction to J. Harrison's theory of \emph{differential chains}\footnote{In the past, these objects were called ``chainlets.''  We no longer use the term, since its definition and meaning were ambiguous.  Though, to be clear, a ``differential chain'' is not necessarily differentiable.  Rather, the term ``differential chain'' is used in reference to the cochains in the theory, which are none other than differential forms.}.  In the second part, we apply these tools to generalize the Cauchy theorems in complex analysis.  Instead of requiring a piecewise smooth path over which to integrate, we can now do so over non-rectifiable curves and divergence-free vector fields supported away from the singularities of the holomorphic function in question.  In the third part, we focus on applications to dynamics, in particular, flows on compact Riemannian manifolds.  We prove that the \emph{asymptotic cycles} are differential chains, and that for an ergodic measure, they are equal as differential chains to the differential chain associated to the vector field and the ergodic measure.  The first part is expository, but the second and third parts contain new results.

\subsection{Background}
Differential chains are best thought of as domains of integration that behave well with calculus.  That is, given some differential form $\o$, we want to know the answer to  the following question: what kind of domains $A$ can we integrate $\o$ over to get an integral $\int_A \o$ such that the theorems of calculus, in particular the Divergence and Stokes' theorems, hold?  Classically, we can integrate only over smooth orientable (sub)manifolds.  However, it is possible to do much better.  

If we are given a topological space of differential forms, the dual pairing $\phi(\o)$ where $\phi$ is an element of the continuous dual space, called \emph{currents}\footnote{Strictly speaking, ``currents'' refers only to the continuous dual space of $C^\i$ forms with compact support when given the topology of uniform convergence in all directional derivatives of all orders.  We use the term more broadly here, to mean the continuous dual space to \emph{any} topological space of differential forms.  We write ``de Rham currents'' to mean currents in the strict sense.}, yields a tautological integral, $\int_\phi \o:=\phi(\o)$.  One then defines ``boundary'' on these domains to be the dual operator to exterior derivative, in which case Stokes' theorem holds by definition.  This was the approach taken by de Rham \cite{derham1,derham2}.  However, it is not clear which topological space of forms we should use, since each yields different spaces of currents.  Moreover, currents tend to have bad topological properties.  If the topological space of forms is not given by a norm, then the strong topology on currents is hard to work with.  This is the case with Schwartz distributions, since they constitute the continuous dual space of smooth functions with compact support, which is not a Banach space.

Even more problematic, however, is the question of regularity.  ``Nice'' geometric objects like piecewise smooth orientable submanifolds should be currents, but what happens when we take a Cauchy sequence of such elements?  What does the resulting object ``look'' like?  Do currents have any geometric meaning beyond their purely topological definition?  What is required is a \emph{representation theorem}\footnote{That is, in the sense of the Riesz representation theorem: in the case of a Hilbert space $X$, one is able to define the continuous dual space as the set of all operators $\{\<\cdot, x\> : x\in X\}$.  We would like to do the same thing in a more general setting.}.  We would like to be able to define the domains separately as geometric objects, and then to give a natural isomorphism from a space of forms to the continuous dual space of the domains.  Such a representation theorem would yield actual ``theorems'' (such as that of Stokes) rather than simply consequences from duality as is the case with de Rham currents.

Attempts have been made to solve these problems by looking at certain subspaces of currents.  In particular, the modern field of geometric measure theory (GMT) was built around the \emph{normal} and \emph{integral} currents of Federer and Fleming \cite{federerfleming,federer,morgan}.  Unfortunately, normal and integral currents fail to incorporate many interesting examples.  

The first example is the Dirac delta distribution.  Given a point $p$, one may define a current $\d_p$, the Dirac delta distribution, which when paired with a function $f$ yields $\int_{\d_p}f:=\<\d_p,f\>=f(p)$.  Analogously, the distributional derivative of $\d_p$, call it $\d_p'$, takes the value $\int_{\d_p'}f:=\<\d_p',f\>=(\p f/\p x)(p)$ when $f$ is smooth.  Neither $\d_p$ nor $\d_p'$ are normal or integral currents however, as one can check (see \cite{federerfleming} for the definition of normal and integral currents).  Essentially, the current $\d_p'$ fails to have finite \emph{mass}.  The currents $\d_p$ and $\d_p'$ are singularly supported, but there are many examples of non-pathological currents that are not, yet still have infinite mass.  Take, for example, the current $\hat{S}$ that takes the value $\hat{S}(\o)=\int_{S} \mathcal{L}_{\p /\p e_1}\o$, where $S$ is the unit circle in $\R^2$ and $\mathcal{L}$ is Lie derivative.  We call these domains ``dipoles'' since one may define them as limits of differences of currents with finite mass.  For example, one may define $\hat{S}$ as $\lim_{t\rightarrow 0} \frac{1}{t}(S_{(te_1,0)}-S_0)$, where $S_v$ denotes the unit circle centered at $v\in \R^2$.  

This leads us to the second example on which normal and integral currents fail to be inclusive, and that is soap films.  Federer and Fleming used normal and integral currents to solve in \cite{federerfleming} a version of Plateau's problem, the general problem being: given a closed wire (a $1$-dimensional submanifold of $\R^3$, modified to allow branching), is there a spanning surface (a \emph{soap film}, possibly non-orientable and with branchings) with minimal area bounded by the wire?  Federer proved a regularity theorem in \cite[Chapter 5.3.17]{federer} which states loosely that the \emph{rectifiable} current of minimal area is a smooth embedded manifold.  However, this precludes possible spanning surfaces with self-intersection, as well as those that are non-orientable. As per the definition of rectifiable currents in \cite{federerfleming}, such currents are required to have finite mass. It turns out that non-rectifiable domains, specifically \emph{dipole} domains play a large role in Harrison's existence theorem, which guarantees a minimizer in this more broad context.  

That said, we can also describe a great many \emph{rough} domains if we drop the condition of finite mass.  The condition eliminates many examples supported on fractals, such as the Mandelbrot set and the Weierstrass nowhere differentiable function, but also topological examples, such as the topologist's sine circle.  Thus, normal and integral currents have limitations at both extremes of smoothness, soap films being objects which are in a sense ``as smooth as possible.''  

Finally, normal and integral currents do not satisfy our goal of being defined independently from differential forms, as one can check by looking at the definition in \cite{federerfleming}.  They do not have a representation theorem as above.

Another approach to the general problem of finding a nice space of domains was tried by Whitney \cite{whitney}.  Whitney's idea was to start with a topological space whose continuous dual is a space of forms.  The integral is thus instead defined as $\int_A \o:=\o(A)$.  He considered two norms on \emph{polyhedral chains} (essentially formal sums of simplices in $\R^n$), called the sharp and the flat norms.  This approach solves the representation problem, since Whitney defines his chains independently from forms, and gave an isomorphism from a space of forms to the dual of polyhedral chains.  However, Whitney's theory has several severe limitations.  Whitney could not prove the divergence theorem in either the flat or sharp topologies.  In particular, problems arise with his norms since the \emph{boundary operator} is not continuous in the sharp norm and \emph{geometric Hodge star} and \emph{linear change of variables} are not continuous in the flat norm \cite{harrison8}.

Harrison's \emph{differential chains} \cite{harrison6,harrison8}\footnote{\cite{harrison6,harrison8} are the most up-to-date references.  For a general history of the subject, see \cite{harrison1,harrison2,harrison3,harrison4,harrison5,harrison7,harrison9}.} solve these problems.  Her approach bypasses currents, by focusing as Whitney did on a space whose dual was a space of differential forms.  However instead of a pair of norms, a whole family of norms is used, under which all the operators of calculus are well-defined and continuous.  The space is constructed as an inductive limit of completions of ``pointed chains'' with respect to these norms, pointed chains being infinitesimal versions of polyhedral chains.  These have the benefit of being singularly supported and as such are much easier to work with.  In particular, the structure of pointed chains allows us to easily transition to an ambient abstract manifold.  Furthermore, differential chains are not required to have finite mass, and this allows us to work with dipoles (see Example \ref{der2}) and non-rectifiable curves.  

Lastly, it is possible to axiomatize the space of differential chains and its topology, since in general the continuity of the fundamental operators (see section \ref{operators}) is what is important, not the topology itself.  However, since we will be working with examples, the constructive version is more pertinent to this discussion.  In what follows, we will give a definition of the space of differential $k$-chains $\hB_k^\i$ on $\R^n$ and its topology, after which we will define the operators used in the application sections.  We will then generalize the theory to manifolds, and finally, we will establish correspondences between classical domains of integration and differential chains.

\section{Preliminaries}
Our goal in this section is to provide a relatively quick, non-exhaustive, yet self-contained introduction to the theory of differential chains.  We will use tools from many parts of the theory in the two applications sections, and as such, it will be necessary to give a somewhat broad introduction.  

\subsection{The Space of Differential Chains $\hB_k^\i$ and its Topology}

\begin{defn}
	Let $\P_k(\R^n)$ denote the space of finitely supported sections of $\L^k T \R^n$.  We call $\P_k(\R^n)$ the space of \textbf{pointed $k$-chains} on $\R^n$.  The set $\P_k(\R^n)$ has a natural vector space structure induced by the vector space structure on sections of $\L^k T \R^n$, denoted $\G(\L^k T \R^n)$.  Let $A\in \P_k(\R^n)$.  We write $A$ in formal sum notation,
	\[
	A=\sum_{i=1}^N (p_i;\a_i),
	\]
	where $\a_i\in \L^k T_{p_i} \R^n$ is the value of $A$ at $p_i$.  If $A=(p;\a)$ and $\a$ is a simple (decomposable) $k$-vector, then we say $A$ is a \textbf{simple $k$-element}, or that $A$ is \textbf{simple}.  Note that strictly speaking, each $p_i$ is allowed to appear in the formal sum only once.  We relax this condition for the sake of notation and allow the same point to appear in the sum $\sum (p_i;\a_i)$ any number of times.  We also write $\P_k$ in place of $\P_k(\R^n)$ if the specific underlying space is not important.
\end{defn}
\begin{figure}[htbp]
	\centering
		\includegraphics[height=3in]{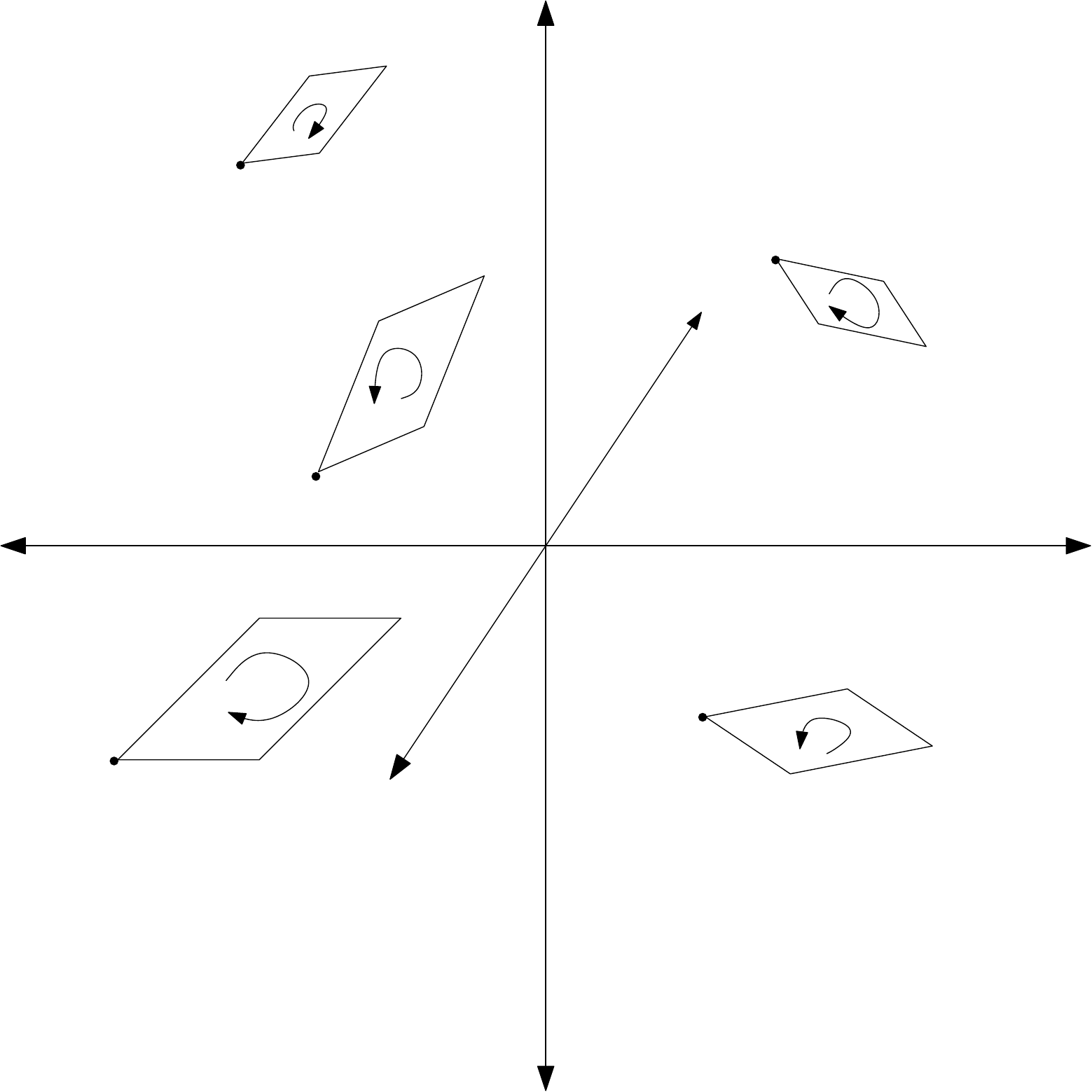}
	\caption{A Pointed Chain in $\P_2(\R^3)$, supported at $5$ points.  The oriented parallelograms here represent $2$-vectors.}
	\label{fig:pointedchain}
\end{figure}

The space $\P_k$ of pointed chains is the foundation upon which the rest of the theory rests.  We will define a sequence of norms $\|\cdot\|_{B^r}$ on $\P_k$, each smaller than the next.  Upon completion, we will get a sequence of Banach spaces, denoted $\hB_k^r$, the space of \emph{differential $k$-chains on $\R^n$ of order $r$}.\footnote{The reason for this notation is that the continuous dual space of $\hB_k^r$ is very closely related to the space $\B^r$, consisting of bounded $C^r$ functions with bounds on the derivatives up to order $r$.  This space was studied by Schwartz as one whose dual was that of ``summable distributions'' in \cite{schwartz1,schwartz2,schwartz3}.  We use similar notation here for the sake of familiarity.}  This will yield an inductive limit of topological vector spaces, which when given the final topology, will be called $\hB_k^\i$, the space of \emph{differential $k$-chains on $\R^n$}.

Note that it is possible to start with pointed chains in a manifold instead of in Euclidian space, however operators like translation cease to be commutative and the task becomes more difficult.  Instead, we first work in Euclidian space, and then move to manifolds using analogous definitions of the norms.

The idea behind pointed chains is vaguely reminiscent of a consequence of the Banach-Alaoglu theorem \cite[\S 3.15, p.68]{rudin} and the Krein-Milman theorem \cite{kreinmilman}, or \cite[p.179]{royden}, which imply in particular that convex combinations of Dirac measures are dense in the set of probability measures given the vague topology \cite{federer}.  Pointed chains are a generalization of convex combinations of Dirac measures, and we will explicitly define a topology in which they are dense.

\begin{lem}\label{pointed gen}
	The set of simple $k$-chains generates $\P_k$.
\end{lem}
\begin{proof}
	This follows from the fact that each $\a_i$ is the sum of simple $k$-vectors.  Indeed, let $A=\sum_{i=1}^N (p_i;\a_i)\in \P_k$ be a sum of simple $k$-chains and let $\a_i=\sum_{j=1}^{N_i}\a_i^j$ where each $\a_i^j$ is simple.  It follows that
	\[
	A=\sum_{i=1}^N  \left(p_i; \sum_{j=1}^{N_i}\a_i^j\right)=\sum_{i=1}^N \sum_{j=1}^{N_i} (p_i; \a_i^j).
	\]
	Since each $(p_i;\a_i^j)$ is a simple $k$-chain, we are done.
\end{proof}

\begin{defn}
	Fix an inner product\footnote{One can show that the resulting norms yield equivalent topologies under different inner products.} $\<\cdot,\cdot\>$ on $\R^n$.  The \textbf{mass norm} or \textbf{$0$-norm} $\|\cdot \|_{B^0}$ on $\P_k$ is given by
	\[
	 \| A\|_{B^0}=\inf\left\{\sum_{i=1}^N \|\a_i\| : A=\sum_{i=1}^N(p_i;\a_i)\right\},
	\]
	where the infimum is taken over all possible ways to write $A=\sum_{i=1}^N(p_i;\a_i)$ and $\|\a_i\|$ is the mass norm\footnote{Throughout, the notation $\|\a_i\|$ will refer to the mass norm.} of $\a_i\in\L^k T_{p_i} \R^n$ induced by $\<\cdot,\cdot\>$.  Recall the mass of a $k$-vector $\a$ is given by 
	\[
	\|\a\|:=\inf \{\sum|\a_i| : \a=\sum \a_i\},
	\]
	where the $\a_i$ are simple, and $|\a_i|$ denotes the $k$-volume of the parallelepiped associated to $\a_i$.  Note that the infimum in the definition of $\|\cdot\|_{B^0}$ is almost a tautology - by the triangle inequality on the mass norm on $k$-vectors, the infimum is achieved when each $p_i$ appears only once in the sum.  Also note that we could have required the $\a_i$'s in the definition of $\|\cdot\|_{B^0}$ to be simple.  The two definitions are equivalent.  
\end{defn}

\begin{lem}
	The function $\|\cdot\|_{B^0}: \P_k\rightarrow \R$ is indeed a norm on $\P_k$.
\end{lem}
\begin{proof}
	We first check positive definiteness.  If $A=0$, then clearly $\| A\|_{B^0}=0$.  Now, suppose $\|A\|_{B^0}=0$ and $A=\sum_{i=1}^N (p_i;\a_i)$ where each $p_i$ appears only once in the sum.  Then, $\|A\|_{B^0}=\sum_{i=1}^N\|\a_i\|=0$, whereby $\|\a_i\|=0$ for each $1\leq i\leq N$.  This implies that $\a_i=0$ for each $1\leq i \leq N$, hence $A=0$.
	
	Now, let $A\in \P_k$.  Write $A=\sum_{i=1}^N(p_i;\a_i)$ where each $p_i$ appears only once.  Thus, $\|A\|_{B^0}=\sum_{i=1}^N \|\a_i\|$.  If $\l\in \R$, then $\l A=\sum_{i=1}^N(p_i;\l \a_i)$ and $\|\l A\|_{B^0}=\sum_{i=1}^N \|\l \a_i\|=\sum_{i=1}^N |\l|\|\a_i\|=|\l|\|A\|_{B^0}$.
	
	Finally, if $A, B\in \P_k$, where $A=\sum_{i=1}^N(p_i;\a_i)$ and $B=\sum_{j=1}^M(q_j; \b_j)$ where each $p_i$ and $q_j$ appear once in their respective sums, then,
	
\begin{align}
	\|A+B\|_{B^0}&=\inf\left\{\sum_{h=1}^L \|\gamma_h\| : A+B=\sum_{h=1}^L(r_h;\gamma_h)\right\}\\
	&\leq \sum_{i=1}^N \|\a_i\|+\sum_{j=1}^M \|\b_j\|= \|A\|_{B^0}+\|B\|_{B^0}.
\end{align}

\end{proof}

The $0$-norm is useful in some situations, but we will need some additional structure to define the higher order norms, as motivated by the following example.

\begin{example}\label{der}
	For each $h>0$, let $A_h=\frac{1}{h}((h e_1 ; e_2)-(0; e_2))=\left(h e_1 ; \frac{e_2}{h}\right)-\left(0;\frac{e_2}{h}\right)\in \P_1(\R^2)$, where $e_1$ and $e_2$ are the unit coordinate vectors.  Let $\o\in \O^1(\R^2)$.  Then we can evaluate $\o$ on $A_h$ in the following manner,
	\[
	\o(A_h)=\o_{h e_1}\left(\frac{e_2}{h}\right)-\o_0 \left(\frac{e_2}{h}\right)=\frac{1}{h}(\o_{h e_1}(e_2)-\o_0 (e_2)),
	\]
	where $\o_p(v)$ means evaluate $\o$ at $p$ on $v\in T_p \R^2$.
	If we write $\o=df$, then this quantity is given by
	\[
	\frac{1}{h}\left(\frac{\p f}{\p e_2}(h e_1)-\frac{\p f}{\p e_2}(0)\right).
	\]
	So, as $h\rightarrow 0$, it follows that $\o(A_h)$ converges to $\frac{\p^2 f}{\p e_1 \p e_2}(0)$.  As such, we would like $\{A_{1/m}\}_{m\in \N}$ to be a Cauchy sequence in our space of differential chains.  However, this sequence diverges to $+\i$ in the $0$-norm, as one can easily verify.    
	\begin{figure}[htbp]
		\centering
			\includegraphics[height=1.5in]{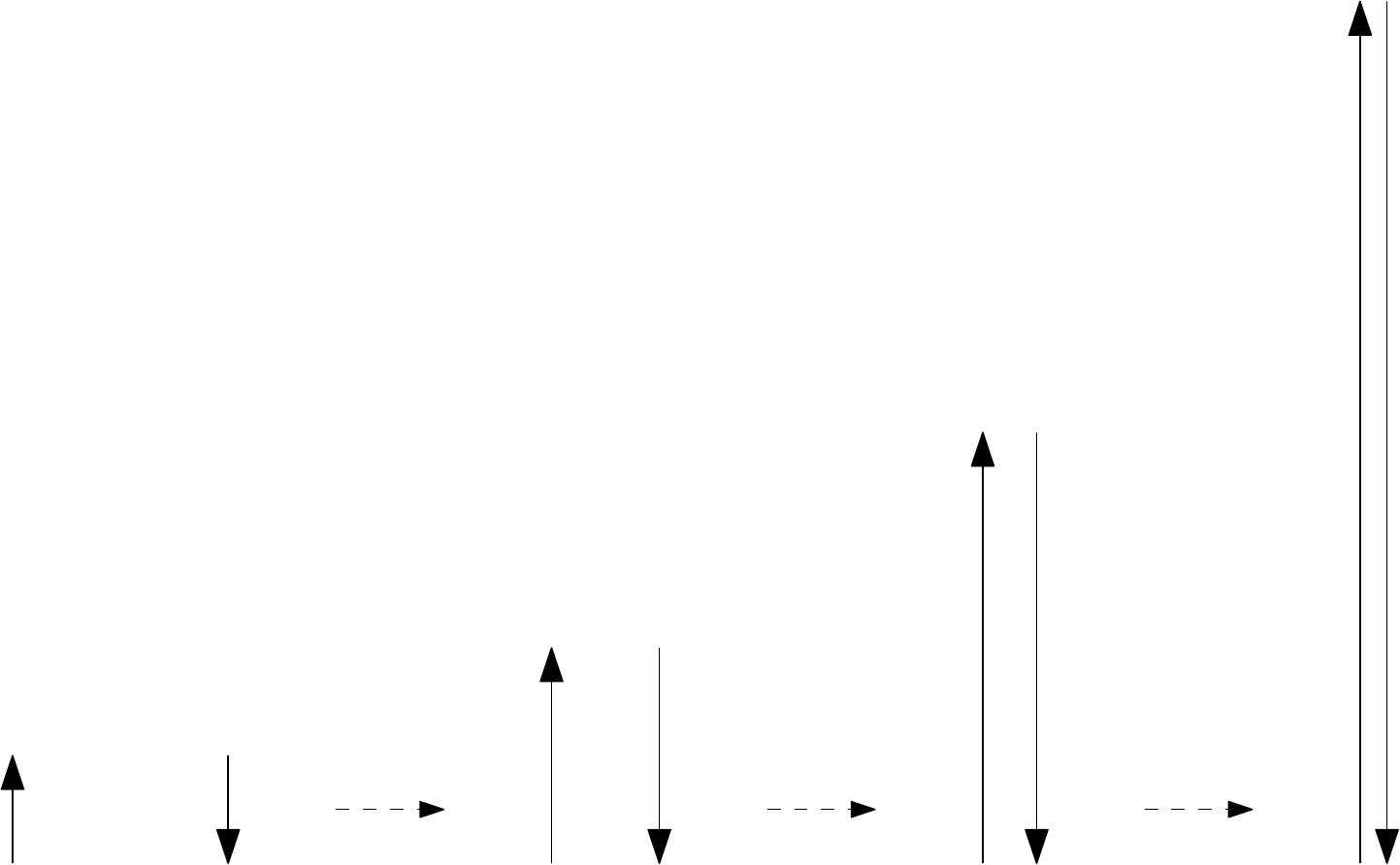}
		\caption{The sequence $\{A_h\}_h$ diverges in the $0$-norm, but limits to a ``dipole.''}
		\label{fig:dipole}
	\end{figure}
\end{example}

\begin{defn}\label{transl}
	Let $u\in \R^n$ and let $T_u: \P_k\rightarrow \P_k$ be the operator given by
	\[
	T_u \sum_{i=1}^N (p_i;\a_i)=\sum_{i=1}^N (p_i+u;\a_i).
	\]
	Likewise, let $\D_u: \P_k\rightarrow \P_k$ be the operator given by
	\[
	\D_u \sum_{i=1}^N (p_i;\a_i)=(T_u-Id)\sum_{i=1}^N (p_i;\a_i)=\sum_{i=1}^N (p_i+u;\a_i)-\sum_{i=1}^N (p_i;\a_i),
	\]
	where $Id$ denotes the identity map.  We call $T_u$ the \textbf{translation operator} and $\D_u$ the \textbf{difference operator}.
\end{defn}

\begin{lem}
	The operators $T_u$ and $\D_u$ are linear and satisfy
	\begin{align}
		T_{u_1}\circ T_{u_2}&=T_{u_2}\circ T_{u_1}=T_{u_1+u_2},& \D_{u_1}\circ \D_{u_2}&=\D_{u_2}\circ \D_{u_1}.
	\end{align}
\end{lem}

\begin{proof}
	Linearity is immediate, as is the commutativity of of $T_u$.  We prove the last equality.
	\begin{align}
		\D_{u_1}\D_{u_2}\sum_{i=1}^N (p_i;\a_i)&=	\D_{u_1}\sum_{i=1}^N (p_i+u_2;\a_i)-\D_{u_1}\sum_{i=1}^N (p_i;\a_i)\\
		&=\sum_{i=1}^N (p_i+u_1+u_2;\a_i)-\sum_{i=1}^N (p_i+u_2;\a_i)-\sum_{i=1}^N (p_i+u_1;\a_i)+\sum_{i=1}^N (p_i;\a_i),
	\end{align}
	whereby we are done, since this value is symmetric in $u_1$ and $u_2$.
\end{proof}

Since $\D_u$ is commutative, we will write $\D_U^j$ to mean $\D_{u_1}\circ\cdots\circ\D_{u_j}$ where $U=\{u_1,\dots,u_j\}$ is a $j$-length index set of vectors, possibly repeating, in $\R^n$.  If $U$ is empty, let $\D_U^0$ denote the identity map on $\P_k$.  Geometrically speaking, the operator $\D_U^j$ turns a simple pointed chain $(p;\a)$ into a possibly degenerate parallelepiped with oriented copies of $\a$ at each vertex.  
\begin{figure}[htbp]
	\centering
		\includegraphics[height=3in]{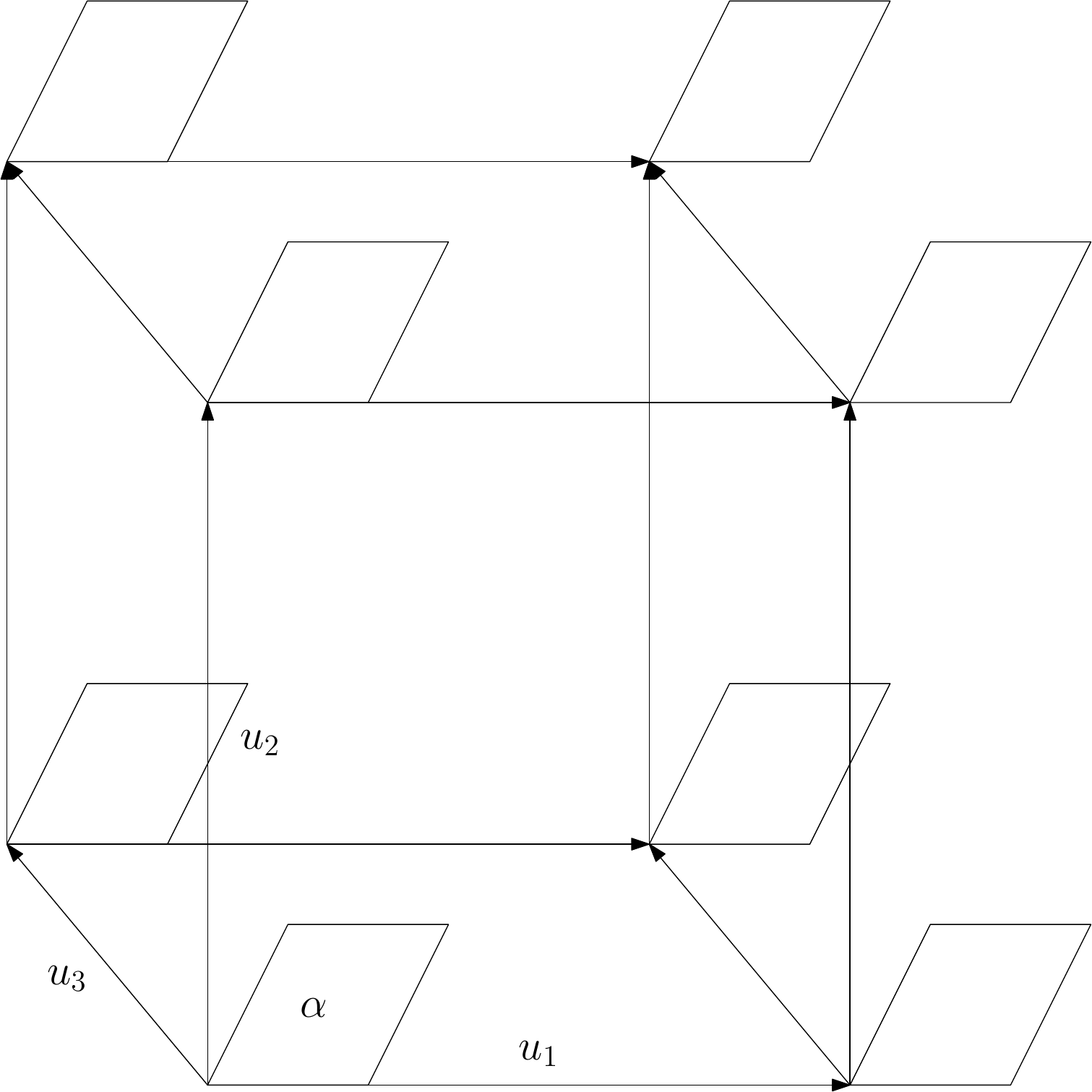}
	\caption{A Difference Chain $\D_{\{u_1,u_2,u_3\}}^3(p;\a)$}
	\label{fig:differencechain}
\end{figure}

\begin{defn}\label{seminormdef}
	If $U$ is empty, define $|\D_U^0(p;\a)|_0:=\|\a\|$.  If $j\geq 1$, and $U=\{u_1,\dots,u_j\}$, define
	\[
	|\D_U^j(p;\a)|_j:=\|u_1\|\cdots \|u_j\|\|\a\|.
	\]
\end{defn}

We are now ready to define the higher order norms on $\P_k$.  	

\begin{defn}\label{normdef}
	For each $r\geq 1$, the \textbf{$r$-norm}\footnote{The idea for these norms first appears in \cite{harrison1} and is further developed in \cite{harrison2}.  The modern presentation using \emph{pointed chains} can be found in \cite{harrison8}.} $\|\cdot\|_{B^r}$ on $\P_k$ is given by
	\[
	\|A\|_{B^r}:=\inf\left\{\sum_{i=1}^N |\D_{U_i}^{j_i}(p_i;\a_i)|_{j_i} : A=\sum_{i=1}^N \D_{U_i}^{j_i}(p_i;\a_i)\right\},
	\]
	where $0\leq j_i\leq r$, $U_i$ is of size $j_i$, and the infimum is taken over all possible ways of writing $A$ in such a way.  Note that if in fact we set $r=0$, this definition also gives the mass norm.  Also note that we do not require the $(p_i;\a_i)$'s to be simple, though to do so would yield an equivalent definition, by the definition of the mass norm on $k$-vectors.
\end{defn}

\begin{example}If $A=\sum_{i=1}^4 (-1)^{i+1}(p_i;\a)$ and the $p_i$ are the vertices of the parallelogram $(p_1, p_1+u, p_1+u+v, p_1+v)$, then $\|A\|_{B^2}$ is bounded above by $\|u\|\|v\|\|\a\|$, as well as $2\|u\|\|\a\|$, $2\|v\|\a\|$, and $4\|\a\|$.  
\end{example}

\begin{lem}
	For each $r\geq 1$, the function $\|\cdot\|_{B^r}:\P_k\rightarrow \R$ is a semi-norm on $\P_k$.
\end{lem}
\begin{proof}
	This is not hard to see: positive homogeneity follows from substitution, subadditivity follows from taking infimums.  
\end{proof}

\begin{lem}\label{decreasing}
	If $r\leq s$, then $\|A\|_{B^s}\leq \|A\|_{B^r}$ for all $A\in \P_k$. 
\end{lem}
\begin{proof}
	That $r\leq s$ implies that if $A=\sum_{i=1}^N \D_{U_i}^{j_i}(p_i;\a_i)$ where $0\leq j_i\leq r$, it also holds that $A=\sum_{i=1}^N \D_{U_i}^{j_i}(p_i;\a_i)$ where $0\leq j_i\leq s$.  It follows that $\|A\|_{B^s}\leq \|A\|_{B^r}$.
\end{proof}

\begin{thm}\label{norm}
	For each $r\geq 1$, the $r$-semi-norm $\|\cdot\|_{B^r}$ is indeed a norm.
\end{thm}
\begin{proof}
	The proof is somewhat involved.  See the appendix, section \ref{normproof}.
\end{proof}

\begin{example}\label{der2}
	In Example \ref{der}, the sequence $A_{1/m}$ is unbounded in the $0$-norm, however, it is easy to see that it is bounded in the $1$-norm.  In fact, as we will show in Lemma \ref{der3}, the sequence is Cauchy in the $2$-norm.	
\end{example}

\begin{lem}\label{der3}
	Let $(p;\a)\in \P_k$ and let $v\in \R^n$.  The sequence $Q_t:=\D_{v/t}(p;t\a)$ is Cauchy in the $2$-norm.
\end{lem}
\begin{proof}
	Using a telescoping sequence, we may write
	\begin{align}
		&\,\,\,\,\,\,\,\left\|\left[(p+2^{-i}v;2^i\a)-(p;2^i\a)\right]-\left[(p+2^{-(i+j)}v;2^{i+j}\a)-(p;2^{i+j}\a)\right]\right\|_{B^2}\\
		&=\left\|\sum_{m=1}^{2^j}\left[(p+m2^{-(i+j)}v;2^i\a)-(p+(m-1)2^{-(i+j)}v;2^i\a)-(p+2^{-(i+j)}v;2^i\a)+(p;2^i\a)\right]\right\|_{B^2}\\
		&=\left\|\sum_{m=1}^{2^j} \D_{(m-1)2^{-(i+j)}v}\D_{2^{-(i+j)}v}(p;2^i\a)\right\|_{B^2}\\
		&\leq \sum_{m=1}^{2^j}\left\|\D_{(m-1)2^{-(i+j)}v}\D_{2^{-(i+j)}v}(p;2^i\a)\right\|_{B^2}\\
		&\leq \sum_{m=1}^{2^j}\left|\D_{(m-1)2^{-(i+j)}v}\D_{2^{-(i+j)}v}(p;2^i\a)\right|_2\\
		&=\sum_{m=1}^{2^j} (m-1)2^{-(i+j)}\|v\|\cdot 2^{-(i+j)}\|v\| \cdot 2^i \cdot \|\a\|\\
		&=2^{-i-2j}\|v\|^2\|\a\|\sum_{m=1}^{2^j-1} m\\
		&\leq 2^{-i}\|v\|^2\|\a\|.
	\end{align}
\end{proof}

\begin{defn}
	Let $\hB_k^r$ denote the metric space completion of $\P_k$ with respect to the $r$-norm for each $r\geq 0$.  It follows that $\hB_k^r$ is a Banach space for every $r\geq 0$.  Elements of $\hB_k^r$ are called \textbf{differential $k$-chains of order $r$}.  As with differential forms, we sometimes drop the term ``differential,'' and simply write ``$k$-chain'' to mean ``differential $k$-chain.''
\end{defn}

\begin{lem}\label{dirlimit}
	The identity map on $\P_k$ extends uniquely to a continuous linear map $\phi_{r,s}: \hB_k^r\rightarrow \hB_k^s$ whenever $r\leq s$ that satisfies
	\begin{enumerate}
		\item[(1)] $\phi_{r,r}=Id$,
		\item[(2)] $\phi_{s,t}\circ\phi_{r,s}=\phi_{r,t}$ for all $r\leq s\leq t$.
	\end{enumerate}
\end{lem}
\begin{proof}
	This follows immediately from Lemma \ref{decreasing}.
\end{proof}

Therefore, it follows that the spaces $\hB_k^r$ along with the maps $\phi_{r,s}$ form a directed system of topological vector spaces, and as such, we get an inductive limit in the category of topological vector spaces \cite[II.29]{bourbaki}.

\begin{defn}
Let $\hB_k^\i$ denote the inductive limit of $\hB_k^r$.  That is,
\[
\hB_k^\i=\lim_{\rightarrow} \hB_k^r.
\]
We give $\hB_k^\i$ the \emph{final topology}, making $\hB_k^\i$ a locally convex topological vector space \cite[\S 19.3]{kothe}.  The final topology is by definition the finest locally convex topology on $\hB_k^\i$ such that the canonical mappings $\psi_r: \hB_k^r\rightarrow \hB_k^\i$ are continuous.  We call $\hB_k^\i$ the space of \textbf{differential $k$-chains on $\R^n$}.  As the inductive limit of Banach spaces, it follows immediately that $\hB_k^\i$ is a \emph{Mackey space} \cite[Ch.V \S 2 Prop. 7]{robertson}.  One can also show that $\hB_k^\i$ is Hausdorff.
\end{defn}

One immediate and useful consequence of the definition of $\hB_k^\i$ is the following:

\begin{lem}
	The space of pointed chains $\P_k$ (or rather strictly speaking, their equivalence classes in the definition of the inductive limit) is dense in $\hB_k^\i$. 
\end{lem}
  
\begin{proof}
If $[a]\in\hB_k^\i$, then $a\in\hB_k^r$ for some $r\geq 0$.  We can find a pointed chain approximation $\{p_n\}_n$ to $a$ in $\hB_k^r$.  Since $p_n\rightarrow a$ and $\psi_r$ is continuous, it follows that $[p_n]\rightarrow [a]$.  
\end{proof}

These norms may seem somewhat ad-hoc.  However, as we will see, the dual space $(\hB_k^r)'$ consists of differential forms whose differentiability class increases as $r$ increases, and the operator norm is what we would expect.  One may also recognize the $1$-norm as Whitney's sharp norm \cite{whitney}. They are indeed the same on polyhedral chains.  However, Whitney did not define higher order norms, which are necessary since the boundary operator (see section \ref{operators}) on a chain with finite $r$-norm is only guaranteed to have a finite $r+1$-norm.

\subsection{The Dual Spaces $(\hB_k^r)'$}\label{dual}
The structure of pointed chains allows us to easily move to a dual space of forms, since any exterior $k$-form $\o$ can be naturally evaluated on any pointed $k$-chain $A=\sum_{i=1}^N (p_i;\a_i)$ by setting $\o(A)=\sum_{i=1}^N \o_{p_i}(\a_i)$.  As such, we make the following definition:

\begin{defn}\label{topequiv0}
	If $\o$ is a bounded $k$-form and $j\geq 0$, define
	\[
	|\o|_{B^j}:=\sup \{|\o (\D_U^j(p;\a))| : |\D_U^j (p;\a)|_j=1\}.
	\]
 	Here, we may either require $(p;\a)$ to be simple, or not.  The two definitions are equivalent, as one can check.  For $r\geq 0$, define
	\[
	\|\o\|_{B^r}:=\max\{|\o|_{B^0},\dots,|\o|_{B^r}\}.
	\]
	Let $\B_k^r$ be the subspace of all bounded $k$-forms with $\|\o\|_{B^r}<\i$.  
\end{defn}

\begin{thm} \label{junk}The following hold:
	\begin{enumerate}
		\item For $r\geq 0$, the function $\|\cdot\|_{B^r}: \B_k^r \rightarrow \R_+$ is a norm, and turns $\B_k^r$ into a Banach space.
		\item For $s\geq r\geq 0$, we have natural and continuous inclusions $\iota_{s,r}: \B_k^{s}\rightarrow \B_k^{r}$.\label{inc}
		\item The space $\B_k^0$ is equal to the space of bounded $k$-forms.
		\item The space $\B_k^1\subset \B_k^0$ is the subspace of such forms that are Lipschitz continuous. 
		\item For $r> 1$, it holds that $\o\in\B_k^r$ if and only if $\o\in C^{r-1}$, its $j$-th directional derivatives are bounded for $0\leq j\leq r-1$, and its $(r-1)$-th directional derivatives are Lipschitz continuous.
		\item For $r\geq 0$, the space $\B_k^r$ is naturally isomorphic to $(\hB_k^r)'$, the continuous dual of $\hB_k^r$.  Furthermore, the injection map $\iota_{s,r}$ is the transpose of $\phi_{r,s}$ (in the sense of Lemma \ref{dirlimit}). \label{trans}
		\item \label{topequiv1} For $r\geq 0$ and for all $\o\in \B_k^r$, the norm $\|\o\|_{B^r}$ is equal to $\|\o\|_{\mathrm{Op}}^r$, where $\|\cdot\|_{\mathrm{Op}}^r$ is the operator norm on $\B_k^r$ considered as the dual space to $\hB_k^r$.
		\item The projective limit $\lim_{\leftarrow}\B_k^r$ via the inclusion mappings in \#\ref{inc} is naturally isomorphic (as a vector space) to $(\hB_k^\i)'$.  The space $\lim_{\leftarrow}\B_k^r$ is a Fr\'{e}chet space.
	\end{enumerate}   
\end{thm}

\begin{proof}
	See section \ref{junkproof} in the appendix for a proof to these claims. 
\end{proof}
\begin{defn}
	We denote the space $\lim_{\leftarrow}\B_k^r$ by $\B_k^\i$.  
\end{defn}

Thus, we have a characterization of the dual space of $\hB_k^r$ as a space of differential forms.  Note that the two equivalent norms above, $\|\cdot\|_{B^r}$ and $\|\cdot\|_{\mathrm{Op}}^r$ are defined in relation to differential chains.  There is a third equivalent norm on $\B_k^r$, given independently from differential chains: 
\begin{defn}\label{topequiv2def}
	Fix $r\geq 1$.  For all bounded forms $\o$ and $0\leq j\leq r-1$ define
	\begin{align*}
	|\o|_{C^j}&:=\sup |D^j \o(p;\a)|,\\
	|\o|_{L^r}&:=\sup \lip(D^{r-1} \o),
	\end{align*}
	where the supremums are taken over all points $p\in \R^n$, all $\a\in \L^k T_p \R^n$ with $\|\a\|=1$, and all directional derivatives $D^j$.  If $\o$ is not $j$-times differentiable, set $|D^j \o(p;\a)|=\i$.  The operator $D^0$ is just taken to be the identity.  The quantity $\lip(D^{r-1}\o)$ is the Lipschitz constant of $D^{r-1}\o$.  Recall that the Lipschitz constant of a form $\o$ is equal to the smallest $C>0$ such that 
	\[
	|\o(p+u;\a)-\o(p;\a)|\leq C |u|
	\]
	for all $p,u\in \R^n$ and $\|\a\|=1$.  
	
	Define $\|\o\|_{C^0}=\sup \{|\o(p;\a)| : p\in \R^n, \|\a\|=1\}$.  Also define
	\[
	\|\o\|_{C^{r-1+\lip}}:=\max \{|\o|_{C^0},\dots,|\o|_{C^{r-1}},|\o|_{L^r}\}.
	\]

\end{defn}

\begin{thm}\label{topequiv2}
	The quantities $\|\o\|_{C^0}$ and $\|\o\|_{B^0}$ are equal for all $\o\in \B_k^0$. For $r\geq 1$, the quantities $\|\o\|_{C^{r-1+\lip}}$ and $\|\o\|_{B^r}$ are equal for all $\o\in \B_k^r$.
\end{thm}
\begin{proof}
	See the appendix, section \ref{topequiv2proof}.
\end{proof}

\begin{remark}
	Note that our space $\B_k^r$ would more therefore more accurately be described as $\B_k^{r-1+\lip}$ (i.e., the space of bounded $r-1$ times differentiable $k$-forms, with bounds on the derivatives, whose $(r-1)$-th derivatives are Lipschitz.  Or, more succinctly but less precisely, the space of bounded $r$-times Lipschitz differentiable $k$-forms.)  The classical definition of $\B^r$ is the space of bounded $C^r$ functions with bounds on the derivatives.  We weaken this requirement, needing only that the $(r-1)$-th derivatives be Lipschitz continuous.  Indeed, to be precise, one should systematically replace $\B_k^r$ with $\B_k^{r-1+\lip}$ throughout.  However, for the sake of the reader's eyesight, we instead opt for a slight abuse of notation.  
\end{remark}

We now have three equivalent ways of describing the topology on $\B_k^r=(\hB_k^r)'$.  We can use the operator norm $\|\o\|_{\mathrm{Op}}^r$, we can compute $\|\o\|_{B^r}$ as a supremum over $\D_U^j(p;\a)$'s (definition \ref{topequiv0}), or we can compute $\|\o\|_{C^{r-1+\lip}}$ as a  supremum over directional derivatives and Lipschitz constants (Definition \ref{topequiv2def}).  All three norms are the same by Theorems \ref{junk}.\ref{topequiv1} and \ref{topequiv2}, so from this point onward, we will simply label this norm as $\|\cdot\|_{B^r}$.  The reason for establishing these equivalences is that now, as of Theorem \ref{topequiv2}, we have defined $\hB_k^r$ and its dual space $\B_k^r$ \emph{completely separately}, including their topologies.

As such, we have accomplished our goal of defining a space of domains ($\hB_k^r$) and a space of forms ($\B_k^r$) such that the space of forms is the continuous dual to the space of domains.  Moreover, we have defined topologies on the domains and forms separately, and have shown that they agree in the sense that the topology on the forms is the strong topology with respect to the topology on the domains.  We can now define an integral pairing between the two.

\begin{defn}
	Let $A\in \hB_k^r$ and let $\o\in \B_k^r$.  Define
	\[
	\cint_A \o :=\o(A).
	\]
	We call this integral the \textbf{Harrison integral}.  By construction, it follows that if $\o_i\rightarrow \o$ in the strong topology on $\B_k^r$ (or even the weak-$*$ topology), then $\cint_A \o_i\rightarrow \cint_A \o$.  Likewise, if $A_i\rightarrow A$ in $\hB_k^r$, then $\cint_{A_i}\o\rightarrow \cint_A\o$.  Thus, the Harrison integral is continuous and linear in both the domain and integrand, and jointly continuous under the strong topology on $\B_k^r$.  Note that this integral is actually an integral: when $A$ represents a classical domain $\hat{A}$ (see section \ref{classical},) then the Harrison integral of $\o$ over $A$ is equal to the Riemann integral of $\o$ over $\hat{A}$.
\end{defn}

\begin{thm}\label{equivnorm}
	The $r$-norm $\|\cdot\|_{B^r}$ on $\hB_k^r$ equal to the norm $\|\cdot\|_{B^r}'$ given by,
	\[
	\|A\|_{{B^r}}':=\sup\left\{\left|\cint_A \o\right| : \o\in \B_k^r,\mathrm{ } \|\o\|_{B^r}\leq 1 \right\}.
	\]
\end{thm}

\begin{proof}
	By Hahn-Banach, there is a canonical isometric injection $\hB_k^r\hookrightarrow (\hB_k^r)''$, given by $A\mapsto \phi_A$, where $\phi_A(\o):=\o(A)$ \cite[p.95]{rudin}.  The result follows from Theorem \ref{junk}.\ref{topequiv1}.
\end{proof}

The space $\B_k^\i$ is slightly more subtle.  It is the space of bounded $C^\i$ forms with bounds on the $j$-th derivatives, for $j\geq 1$.  Note that these bounds are \emph{not} uniform over $j$.  There are several topologies on $\B_k^\i$: we can give $\B_k^\i$ the strong topology considered as the continuous dual to $\hB_k^\i$, and we can give $\B_k^\i$ the initial topology as the projective limit $\lim_{\i\leftarrow r} \B_k^r$.  The strong and initial topologies are actually the same, by \cite[Ch.V, \S4, Prop. 15]{robertson}.  Moreover, as a countable projective limit of Banach spaces, the space $\B_k^\i$ is naturally a Fr\'{e}chet space, with (semi)-norms $\|\cdot\|_{B^r}, r\geq 0$.  However, we can describe another topology on $\B_k^\i$, and that is the weak-* topology: we say that $\o_i\rightarrow \o$ in the weak-* topology if $\cint_J \o_i\rightarrow \cint_J  \o$ for all $J\in \hB_k^\i$.

\begin{example}
	The function $\sin(x)\in \B_0^r(\R)$ for all $r\geq 0$, and hence $\sin(x)\in \B_0^\i(\R)$.  The function $\sin(x^2)\notin \B_0^r(\R)$ for $r\geq 1$, since $\frac{\p}{\p x}  \sin(x^2)$ is not bounded.  The function $x$ is not in $\B_0^0(\R)$, since it is not bounded.  Smooth partition of unity functions are, however, elements of $\B_0^\i(\R)$, as is any smooth function with compact support.  
\end{example}

\subsection{Operators on $\hB_k^r$ and Their Dual Operators on $\B_k^r$}\label{operators}
We now define four fundamental continuous linear operators on $\P_k(\R^n)$.  These will extend by continuity to continuous linear operators on $\hB_k^r(\R^n)$. Then, we will use these fundamental operators to build other, more complicated (and interesting) ones.

\subsubsection{Multiplication by a function}

\begin{defn}
	Let $A=\sum(p_i;\a_i)\in \P_k(\R^n)$, $f\in \B_0^r(\R^n)$, $r\geq 0$.  Define
	\[
		m_f A=\sum(p_i; f(p_i)\a_i).
	\]	
	The linear map $m_f: \P_k(\R^n)\rightarrow \P_k(\R^n)$ is called \textbf{multiplication by $f$}.
\end{defn}

\begin{lem}
	The map $m_f$ is continuous in the $r$-norm topology and thus extends to a map $m_f: \hB_k^r(\R^n)\rightarrow \hB_k^r(\R^n)$.  The map $m_f$ is also continuous and linear in $f$: if $f_i\rightarrow f$ in the norm topology on $\B_0^r(\R^n)$, then $m_{f_i} J\rightarrow m_f J$ for all $J\in \hB_k^r(\R^n)$.  As such, $m_{\cdot}$ defines a jointly continuous map on $\B_0^r(\R^n)\times \hB_k^r(\R^n)$.
\end{lem}

\begin{proof}
	Linearity in both variables is immediate from the definition.  We show continuity.  If $f_i\rightarrow 0$ in $\B_0^r(\R^n)$, then 
	\[
	\|m_{f_i}A\|_{B^r}=\sup_{0\neq \o\in \B_k^r(\R^n)}\frac{\cint_A f_i\o}{\|\o\|_{B^r}}\leq \|A\|_{B^r}\|f_i\o\|_{B^r}\rightarrow 0
	\]
	by the product rule on $f_i\o$.
	
	Likewise, 
	\[
	\|m_fA_i\|_{B^r}=\sup_{0\neq \o\in \B_k^r(\R^n)}\frac{\cint_A f \o}{\|\o\|_{B^r}}\leq \|A_i\|_{B^r}\sup \frac{\|f\o\|_{B^r}}{\|\o\|_{B^r}}.
	\]
	By the product rule, there exists some $K_f>0$ such that $\|f\o\|_{B^r}/\|\o\|_{B^r}<K_f$ for all $0\neq \o\in\B_k^r(\R^n)$.  If $\|A_i\|_{B^r}\rightarrow 0$, it follows that $\|m_fA_i\|_{B^r}\rightarrow 0$ as well, and hence $m_f$ is continuous.   
\end{proof}

\begin{lem}
	The dual operator to multiplication by $f$ is \emph{multiplication by $f$}.  That is,
	\[
	\cint_{m_f J}\o=\cint_J f\o
	\]
	for all $J\in \hB_k^r(\R^n)$ and all $\o\in \B_k^r(\R^n)$.  
\end{lem}

\subsubsection{Extrusion}
\begin{defn}
Let $v\in \R^n$ and let $\check{v}$ be the constant vector field on $\R^n$ in the direction of $v$.  Let $A=\sum(p_i;\a_i)\in \P_k(\R^n)$.  Define
\[
E_v(A)=\sum(p_i;\check{v}(p_i)\wedge \a_i).
\]
This map is called \textbf{extrusion through $v$}\footnote{The \emph{extrusion} operator first appears in \cite{harrison4}.}.
\end{defn}

\begin{lem}
	The linear operator $E_v: \P_k(\R^n)\rightarrow \P_{k+1}(\R^n)$ is continuous in the $r$-norm topology for all $r\geq 0$.  It is also linear and continuous with respect to $v$: if $v_i\rightarrow v$ in $\R^n$, then $E_{v_i} J\rightarrow E_v J$ for all $J\in \hB_k^r(\R^n)$.
\end{lem}

\begin{proof}
	It is enough to show that $\|E_v(A)\|_{B^r}\leq \|v\|\|A\|_{B^r}$ for all $A\in \P_k(\R^n)$.  First, note that $E_v(\D_U^j(p;\a))=\D_U^j(p;\check{v}(p)\wedge\a)$.  It follows that for all $r\geq j$,
	\[
	\|E_v(\D_U^j(p;\a))\|_{B^r}\leq |\D_U^j(p;\check{v}(p)\wedge\a)|_j\leq \|v\| |\D_U^j(p;\a)|_j.
	\]
	Let $A\in \P_k(\R^n)$, $r\geq 0$, and $\e>0$.  Then we can write $A=\sum_{i=1}^N \D_U^{j_i}(p_i;\a_i)$ such that $\|A\|_{B^r}>\sum_{i=1}^N |\D_U^{j_i}(p_i;\a_i)|_{j_i}-\e$.  Thus,
	\begin{align}
		\|E_v(A)\|_{B^r}&\leq \sum_{i=1}^N \|E_v(\D_U^{j_i}(p_i;\a_i))\|_{B^r}\\
		&\leq \|v\|\sum_{i=1}^N |\D_U^{j_i}(p_i;\a_i)|_{j_i}\\
		&\leq \|v\| (\|A\|_{B^r}+\e).
	\end{align}
	Since the inequality holds for all $\e>0$, we conclude that $\|E_v(A)\|_{B^r}\leq \|v\|\|A\|_{B^r}$.
\end{proof}

Thus, for each $r\geq 0$, we get a continuous linear operator $E_v: \hB_k^r(\R^n)\rightarrow \hB_{k+1}^r(\R^n)$.  Since $\hB_k^r(\R^n)$ is a Banach space, it follows that $E_{\cdot}$ defines a jointly continuous map on $\R^n\cdot \hB_k^r(\R^n)$.

\begin{lem}
	The dual operator to extrusion is \emph{interior product}.  That is,
	\[
	\cint_{E_v J}\o=\cint_J \iota_{\check{v}} \o,
	\]
	for all $r\geq 0$, all $J\in \hB_k^r(\R^n)$ and all $\o\in \B_{k+1}^r(\R^n)$.  
\end{lem}

\begin{proof}
	This is certainly the case for $J\in \P_k(\R^n)$.  Since $\B_k^r(\R^n)=(\P_k(\R^n), \|\cdot\|_{B^r})'$, it follows that $\iota_{\check{v}}\o\in \B_k^r(\R^n)$.  Thus,  the equality holds by continuity of the integral for all $J\in \hB_k^r(\R^n)$.
\end{proof}

\subsubsection{Retraction}
\begin{defn}
	Let $v, \check{v}$ be as before, and let $A=(p;v_1\wedge\cdots\wedge v_k)\in \P_k(\R^n)$ be a simple $k$-element, $k\geq 1$.  Define
	\[
	E_v^\dagger(A)=\sum_{i=1}^k (-1)^{i+1}\<v,v_i\>(p;v_1\wedge\cdots\wedge \hat{v_i}\wedge\cdots v_k),
	\]
	where $\hat{v_i}$ signifies that $v_i$ is removed from the wedge product.  This extends to a linear map on $\P_k(\R^n)$, called \textbf{retraction}\footnote{The \emph{retraction} operator first appears in \cite{harrison11}.}.
\end{defn}

\begin{lem}
	The operator $E_v^\dagger: \P_k(\R^n)\rightarrow \P_{k-1}(\R^n)$ is well-defined and continuous in the $r$-norm for all $r\geq 0$.  It is also linear and continuous with respect to $v$: if $v_i\rightarrow v$ in $\R^n$, then $E_{v_i}^\dagger J\rightarrow E_v^\dagger J$ for all $J\in \hB_k^r(\R^n)$.  
\end{lem}

\begin{proof}
	It is well known that this type of contraction on a $k$-vector is well-defined.  We show continuity.  It is enough to show that
	\[
	\|E^\dagger_v(A)\|_{B^r}\leq k\|v\|\|A\|_{B^r}.
	\]
	For each $r\geq j$, we know that 
	\[\|E^\dagger_v(\D_U^j(p;\a))\|_{B^r}=\|\D_U^j(E^\dagger_v(p;\a))\|_{B^r}\leq |\D_U^j(E^\dagger_v(p;\a))|_j\leq \|u_1\|\cdots\|u_j\|\|E^\dagger_v(p;\a)\|_{B^0}.\]
	
	We show that $\|E^\dagger_v(p;\a)\|_{B^0}\leq k\|v\|\|\a\|$ when $\a$ is simple.  Let $\{e_1,\dots,e_k\}$ be an orthonormal basis of the $k$-dimensional subspace of $\R^n$ spanned by $\a$.  Then $\a=\l e_1\wedge\cdots\wedge e_k$ for some $\l\in \R$.  Thus, $E^\dagger_v(p;\l e_1\wedge\cdots\wedge e_k)=(p;\l \sum_{i=1}^k (-1)^{i+1}\<v,e_i\>(e_1\wedge\cdots\wedge\hat{e_i}\wedge\cdots\wedge e_k))$.  It follows that $\|E^\dagger_v(p;\a)\|_{B^0}\leq k\|v\| \|\a\|$.  
	
	We conclude that $\|E^\dagger_v(\D_U^j(p;\a))\|_{B^r}\leq k\|v\|\|u_1\|\cdots\|u_j\|\|\a\|=k\|v\||\D_U^j(p;\a)|_j$.  
	
	Let $A\in \P_k(\R^n)$, $r\geq 0$ and $\e>0$.  We write $A=\sum_{i=1}^N \D_U^{j_i}(p_i;\a_i)$ such that $\|A\|_{B^r}> \sum_{i=1}^N |\D_U^{j_i}(p_i;\a_i)|^j-\e$.  It follows that
	\begin{align*}
		\|E^\dagger_v(A)\|_{B^r}&\leq \sum_{i=1}^N \|E^\dagger_v(\D_U^{j_i}(p_i;\a_i))\|_{B^r}\\
		&\leq k\|v\|\sum_{i=1}^N |\D_U^{j_i}(p_i;\a_i)|_{j_i}\\
		&\leq k\|v\|(\|A\|_{B^r}+\e).
	\end{align*}
	The result follows.
\end{proof}

Thus, we get a jointly continuous map $E^\dagger_{\cdot}: \R^n\times \hB_k^r(\R^n)\rightarrow \hB_{k-1}^r(\R^n)$ for all $r\geq 0$, $k\geq 1$.

\begin{lem}
	The dual operator to retraction is \emph{wedge product}.  That is, if $J\in \hB_k^r(\R^n)$, $k\geq 1$, then
	\[
	\cint_{E_v^\dagger J}\o=\cint_J \check{v}^\flat\wedge \o,
	\]
	for all $\o\in \B_{k-1}^r(\R^n)$, where $\check{v}^\flat$ denotes the $1$-form $(p,w)\rightarrow \<\check{v}(p),w\>$.  
\end{lem}

\subsubsection{Prederivative}
\begin{defn}\label{pre}
	Let $v\in \R^n$ and let $(p;\a)\in \P_k(\R^n)$, $k\geq 0$.  Define
	\[
	P_v(p;\a)=\lim_{t\rightarrow 0}\D_{tv}(p;\a/t)=\lim_{t\rightarrow 0}(p+tv; \a/t)-(p;\a/t),
	\]
	and extend linearly to all of $\P_k(\R^n)$.
\end{defn}

\begin{lem}
	The map $P_v$ extends continuously to a map $P_v: \hB_k^r(\R^n)\rightarrow \hB_k^{r+1}(\R^n)$ for all $r\geq 1$.  We call this operator \textbf{prederivative}\footnote{The \emph{prederivative} operator first appears in \cite{harrison4}}.  The map $P_v$ depends linearly and continuously on $v$.  Hence, $P_{\cdot}$ defines a jointly continuous map on $\R^n\times \hB_k^r(\R^n)$.
\end{lem}
\begin{proof}
	By way of Lemma \ref{der3}, one can see that $P_v(p;\a)\in \hB_k^2(\R^n)$ and so $P_v: \P_k(\R^n)\rightarrow \hB_k^2(\R^n)$ is a well-defined linear map.  We show that $P_v$ is linear in $v$: homogeneity follows immediately since $\l(p;\a)=(p;\l\a)$, and additivity reduces to showing
	\[
	\lim_{t\rightarrow 0}(p; t(v_1+v_2);\a/t)-(p+tv_1;\a/t)-(p+tv_2;\a/t)+(p;\a/t)=0
	\]
	in $\hB_k^2(\R^n)$.  This holds since
	\begin{align*}
	\|(p; t(v_1+v_2);\a/t)-(p+tv_1;\a/t)-(p+tv_2;\a/t)+(p;\a/t)\|_{B^2}&=\|\D_{\{tv_1, tv_2\}}^2(p;\a/t)\|_{B^2}\\
	&\leq |\D_{\{tv_1, tv_2\}}^2(p;\a/t)|_{2}\\
	&= t\|v_1\|\|v_2\|\|\a\|.
	\end{align*}
	
	To show that $P_v$ extends to a continuous linear map $P_v: \hB_k^r(\R^n)\rightarrow \hB_k^{r+1}(\R^n)$ and that $P_v$ is continuous in $v$, it is enough to show the inequality
	\[
	\|P_v A \|_{B^{r+1}}\leq \|v\| \|A\|_{B^r}
	\]
	 for all $v\in \R^n$ and all $A\in \P_k(\R^n)$. For all $j\leq r$, we may write 
	\begin{align}\label{lastline}
		\|P_v(\D_U^j(p;\a))\|_{B^{r+1}}=\lim_{t\rightarrow 0}\|\D_{tv}\D_U^j(p;\a/t)\|_{B^{r+1}}\leq\lim_{t\rightarrow 0}|\D_{tv}\D_U^j(p;\a/t)|_{j+1}=\|v\||\D_U^j(p;\a)|_j.
	\end{align}
	So, let $A\in \P_k(\R^n)$ and let $\e>0$.  By Definition \ref{normdef}, we may write $A=\sum_{i=1}^N\D_{U_i}^{j_i}(p_i;\a_i)$ such that
	\[
	\|A\|_{B^r}>\sum_{i=1}^N|\D_{U_i}^{j_i}(p_i;\a_i)|_{j_i}-\e.
	\]
	By the triangle inequality and (\ref{lastline}), it follows that
	\begin{align}
		\|P_v A\|_{B^{r+1}}\leq \sum_{i=1}^N \|P_v\D_{U_i}^{j_i}(p_i;\a_i)\|_{B^{r+1}}\leq \|v\|\sum_{i=1}^N \|\D_{U_i}^{j_i}(p_i;\a_i)|_{j_i}\leq \|v\| (\|A\|_{B^r}+\e).
	\end{align}
\end{proof}

\begin{lem}
	The dual operator to prederivative is Lie derivative.  That is, if $J\in \hB_k^r(\R^n)$ and $r\geq 1$, then
	\[
	\cint_{P_vJ}\o=\cint_J \mathcal{L}_{v} \o,
	\]
	for all $\o\in \B_k^{r+1}(\R^n)$.
\end{lem}

We now use the above operators to generate an \emph{operator algebra}, $\mathcal{A}(\hB_k^\i(\R^n))$.  Some of the more interesting operators in this algebra are listed below.

\subsubsection{Boundary}
\begin{defn}
	For $r\geq 1$ define the map $\p: \hB_k^r(\R^n)\rightarrow \hB_{k-1}^{r+1}(\R^n)$ by setting
	\[
	\p:=\sum_{i=1}^n P_{v_i}\circ E_{v_i}^\dagger,
	\]
	where $\{v_1,\dots,v_n\}$ is a basis of unit vectors.  We call this operator \textbf{boundary}\footnote{The \emph{boundary} operator first appears in \cite{harrison3}.}.  If $\p J=0$, we say that $J$ is \textbf{closed}.
\end{defn}

\begin{lem}\label{stokes}
	The dual operator to boundary is \emph{exterior derivative}.  That is,
	\[
	\cint_{\p J}\o=\cint_J d\o,
	\]
	for all $J\in \hB_k^r(\R^n)$ and all $\o\in \hB_{k-1}^{r+1}(\R^n)$.
\end{lem}

\begin{proof}
	Expanding out, we get
	\[
	\cint_{\p J}\o=\cint_J \sum_{i=1}^n dv_i\wedge \mathcal{L}_{v_i}\o=\cint_J d\o.
	\]
\end{proof}

\begin{lem}
	The map $\p$ is well-defined.  That is, it does not depend on which unit basis we choose.
\end{lem}

\begin{proof}
	This follows since the exterior derivative $d$ does not depend on the basis.
\end{proof}

Once we show that classical domains are represented in $\hB_k^r(\R^n)$, Lemma \ref{stokes} will imply the classical Stokes' theorem.  Lemma \ref{stokes} also shows that $\p\circ\p\equiv 0$.  From this, we may calculate homology classes.  We call this homology theory \emph{differential homology}.  Harrison has shown \cite{harrison10} that differential homology satisfies a slightly modified version of the Eilenberg-Steenrod axioms \cite{eilenbergsteenrod}.

\subsubsection{Perpendicular Complement}
\begin{defn}
	For $r\geq 0$ define the map $\perp: \hB_k^r(\R^n)\rightarrow \hB_{n-k}^r(\R^n)$ by setting
	\[
	\perp := \prod_{i=1}^n (E_{v_i}+ E_{v_i}^\dagger),
	\]
	where $\{v_1,\dots,v_n\}$ is an orthonormal basis of $\R^n$.  This operator is called \textbf{perpendicular complement}, or just \textbf{perp}\footnote{The \emph{perp} operator first appears in \cite{harrison9}, and can also be found in \cite{harrison6}.}.
\end{defn}

\begin{lem}
	The dual operator to perp is \emph{Hodge star}.  That is, for all $r\geq 0$, $J\in \hB_k^r(\R^n)$, we have
	\[
	\cint_{\perp J}\o = \cint_J * \o,
	\]
	 for all $\o\in \B_{n-k}^r(\R^n)$.
\end{lem}

\begin{proof}
	Expanding the product $\perp = \prod_{i=1}^n (E_{v_i}-E_{v_i}^\dagger)$ to a simple $k$-element $(p; a v_1\wedge\cdots\wedge v_k)$, we get $\perp (p; av_1\wedge\cdots\wedge v_k)=(p; a v_{k+1}\wedge\cdots\wedge v_n)$.  In general, the individual terms of the product expansion applied to a simple $k$-element $(p;\a)$ will be identically zero unless the extrusion operators are not in the $k$-direction of $\alpha$.  If so, the retraction operators \emph{will} be in the $k$-direction of $\alpha$, yielding the perpendicular complement up to the appropriate sign change necessary for Hodge-$*$.  
\end{proof}

\begin{lem}
	The operator $\perp$ is well-defined.  That is, it is independent of our choice of orthonormal basis.
\end{lem}
\begin{proof}
	This follows since $*$ is independent of our choice of orthonormal basis.
\end{proof}

\subsubsection{The Generalized Divergence and Curl Theorems}
We may combine boundary and perp to immediately get general versions of the Divergence and Curl theorems:

\begin{thm}[Divergence Theorem]\label{divergence}
	Let $r\geq 1$ and let $J\in \hB_k^r(\R^n)$.  If $\o\in \B_{n-k+1}^{r+1}(\R^n)$, then
	\[
	\cint_{\perp\p J}\o=\cint_J d*\o.
	\]
\end{thm}

\begin{thm}[Curl Theorem]\label{curl}
	Let $r$ and $J$ be as in Theorem \ref{divergence} and let $\o\in \B_{n-k-1}^{r+1}(\R^n)$.  Then
	\[
	\cint_{\p\perp J}\o=\cint_J * d \o.
	\]
\end{thm}

In fact, we can compose \emph{any} of the above operators to yield similar integral relations.  For example, set $\Diamond=\perp\p\perp$.  Then the dual operator to $\Diamond$ is $\d=*d*$, and we have for all $J\in \hB_k^r(\R^n)$, $r\geq 1$, and $\o\in \B_{k+1}^{r+1}(\R^n)$ the integral relation
\[
\cint_{\Diamond J}\o=\cint_J \d \o.
\]
Likewise, if we define $\Box=\Diamond \p+\p\Diamond$, then the dual operator is $\D=\d d+d\d$, and we have for all $J\in \hB_k^r(\R^n)$, $r\geq 1$, and $\o\in \B_k^{r+2}$ the integral relation
\[
\cint_{\Box J}\o=\cint_J \D \o.
\]

\subsubsection{Operators with respect to a Vector Field}
More generally, we may define the extrusion, retraction, and prederivative operators with respect to a vector field.  Let $V_\B^r(\R^n)$ be the space of vector fields on $\R^n$ whose coordinate functions are elements of $\B_0^r(\R^n)$.  Clearly, the definition of $V_\B^r(\R^n)$ is independent of a choice of basis. 

\begin{defn}[Extrusion with respect to a Vector Field]
	For $X\in V_\B^r(\R^n)$, define
	\begin{align*}
		E_X: \P_k(\R^n)&\rightarrow P_{k+1}(\R^n)\\
		\sum(p_i;\a_i)&\mapsto \sum(p_i; X(p_i)\wedge\a_i).
	\end{align*}
\end{defn}

\begin{thm}\label{before}
	For each $X\in V_\B^r(\R^n)$, $r\geq 0$, the operator $E_X$ is continuous in the $r$-norm topology, and thus extends to a continuous homomorphism $E_X: \hB_k^r(\R^n)\rightarrow \hB_{k+1}^r(\R^n)$.  The dual operator to $E_X$ is $\iota_X$.  That is, for all $J\in \hB_k^r(\R^n)$, $\o\in \B_{k+1}^r(\R^n)$, $X\in V_\B^r(\R^n)$, $r\geq 0$, we have
	\[
	\cint_{E_X J}\o = \cint_J \iota_X \o.
	\]
\end{thm}
\begin{proof}
	We begin by noting that for $f\in \B_0^r(\R^n)$, we have $E_{f\check{v}}(p;\a)=(p;f(p)v\wedge \a)=m_f(p;v\wedge \a)=m_f E_{\check{v}}(p;\a)$.  So, choosing a basis $\{e_i\}_i$ of $\R^n$, we have $X=\sum f_i e_i$, where $f_i\in \B_0^r(\R^n)$.  It follows that $E_X=\sum E_{f_i\check{e_i}}=\sum m_{f_i}E_{\check{e_i}}$.  Since $m_{f_i}$ and $E_{e_i}$ are continuous, it follows that $E_X$ is also continuous.  The fact that $\iota_X$ is the dual operator to $E_X$ follows from the definition of $E_X$ on $\P_k(\R^n)$.  
\end{proof}

\begin{defn}[Retraction with respect to a Vector Field]
	For $X\in V_\B^r(\R^n)$, define
	\begin{align*}
		E_X^\dagger: \P_{k+1}(\R^n)&\rightarrow \P_k(\R^n)\\
		(p; v_1\wedge\cdots\wedge v_{k+1})&\mapsto \sum_{i=1}^{k+1}(-1)^{i+1}\<X(p), v_i\>(p;v_1\wedge\cdots\wedge\hat{v_i}\wedge\cdots\wedge v_{k+1}).
	\end{align*}
\end{defn}
\begin{thm}
	For each $X\in V_\B^r(\R^n)$, $r\geq 0$, the operator $E_X^\dagger$ is continuous in the $r$-norm topology, and thus extends to a continuous homomorphism $E_X^\dagger: \hB_{k+1}^r(\R^n)\rightarrow \hB_k^r(\R^n)$.  The dual operator to $E_X^\dagger$ is $X^\flat\wedge \cdot$.  That is, for all $J\in \hB_{k+1}^r(\R^n)$, $\o\in \B_{k}^r(\R^n)$, $X\in V_\B^r(\R^n)$, $r\geq 0$, we have
	\[
	\cint_{E_X^\dagger J}\o = \cint_J X^\flat\wedge \o.
	\]
\end{thm}
\begin{proof}
	The proof follows in the same manner as in Theorem \ref{before} by noting that $E^\dagger_{f\check{v}}=m_f E^\dagger_{\check{v}}$.  
\end{proof}

\begin{defn}[Prederivative with respect to a Vector Field]
	For $X\in V_\B^{r+1}(\R^n)$, $r\geq 1$, define the operator $P_X:= \p E_X+E_X\p$ on $\hB_k^r(\R^n)$.  By duality and Cartan's magic formula, it follows that the dual operator to $P_X$ is Lie derivative, $\mathcal{L}_X$.
\end{defn}

\begin{lem}
	If $X\in V_\B^{r+1}(\R^n)$, $r\geq 1$, $(p;\a)\in \P_k(\R^n)$, and if $\phi_t$ denotes the time-$t$ flow of $X$, then $P_X(p;\a)=\lim_{t\rightarrow 0}(\phi_t(p);(\phi_t)_* \a/t)-(p;\a/t)$.
\end{lem}
\begin{proof}
	Note that since $r\geq 1$, the vector field $X$ is at least $C^{1+\lip}$, and thus is locally integrable, and the flow map $\phi_t$ is at least of class $C^{1+\lip}$.  Thus, we may form the pushforward map $(\phi_t)_*$ on $k$-vectors in $\R^n$.  The equality follows from the fact that 
	\[
	\o(P_X(p;\a))=\mathcal{L}_X\o(p;\a)=\lim_{t\rightarrow 0}\o(\phi_t(p); (\phi_t)_* \a/t)-\o(p;\a/t)=\o(\lim_{t\rightarrow 0}(\phi_t(p);(\phi_t)_* \a/t)-(p;\a/t)).
	\]
\end{proof}

\subsection{Differential Chains on Open Sets}
In order to move the theory onto manifolds, we need a method for dealing with differential chains in an open set.  That is, instead of the space $\hB_k^r(\R^n)$, it would be useful to have a space $\hB_k^r(W)$ whenever $W$ is an open subset of $\R^n$.  One must be careful about defining $\hB_k^r(W)$, however.  If $\P_k(W)$ denotes the space of pointed chains supported in $W$,  and if we define $\hB_k^r(W)$ to be the completion of $\P_k(W)$ under the $r$-norm, then there are examples that do not behave well under \emph{pushforward} by smooth maps.  For example, if we define $W$ to be a horse-shoe that touches its ends together, we could define a sequence of pointed chains converging to a dipole as in Example \ref{der}, where the dipole ``bridges'' the gap in $W$.  One could then continuously deform $W$ in such a way that would not descend to a continuous pushforward map on $\hB_k^r(W)$.  See Figure \ref{fig:splitting} for an illustration of this.  

  \begin{figure}[htbp]
  	\centering
  		\includegraphics[height=1.5in]{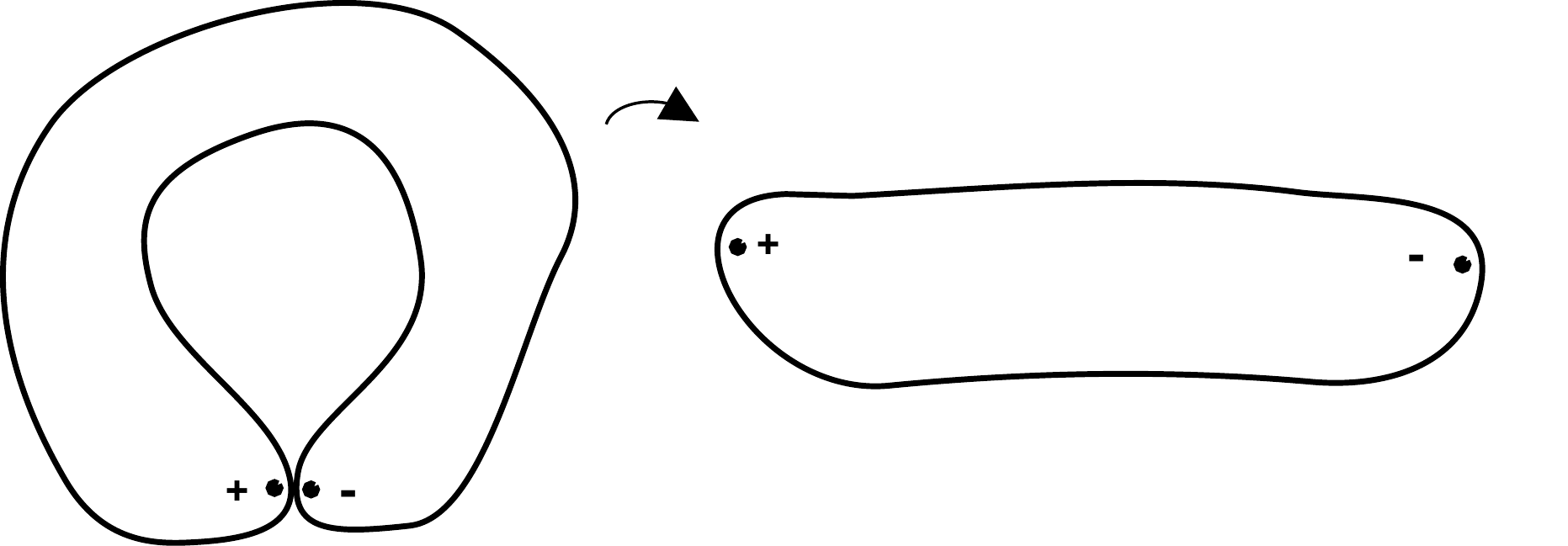}
  	\caption{Discontinuous Pushforward}
  	\label{fig:splitting}
  \end{figure}

To get around this problem, we will need a notion of \emph{support} of a differential chain.  

\begin{defn}
	 If $J\in\hB_k^r(\R^n)$ and $X\subseteq \R^n$ is a closed subset, we say that $X$ \textbf{supports $J$} or \textbf{$J$ is supported by $X$} if, for every non-empty open $W\subseteq \R^n$ with $X\subseteq W$, there exists $A_i\rightarrow J$ in $\hB_k^r(\R^n)$ with $A_i\in \P_k(W)$.  The intersection of all closed sets $X$ supporting $J$ is called \textbf{support} of $J$, and we denote it by $\supp(J)$.  We say that an arbitrary subset $Z\subseteq \R^n$ \textbf{supports} $J$ if $\supp(J)\subseteq Z$.  
\end{defn}
Immediately we see that the support of a pointed chain $A=\sum(p_i;\a_i)$ is $\cup \{p_i\}$, so this definition agrees with the notion of support for pointed chains as finitely supported sections of $\L^k TM$.

\begin{lem}\label{supportfacts}
	The following facts about support hold:
	\begin{enumerate}
		\item If $J\in \hB_k^r(\R^n)$, then $\supp(J)$ is a well-defined closed set.  If $J\neq 0$, then $\supp(J)$ is non-empty.
		\item If $J,K\in \hB_k^r(\R^n)$, then $\supp(J+K)\subseteq \supp(J)\cup \supp(K)$, with equality if $\supp(J)\cap \supp(K)=\emptyset$.  
		\item If $J\in \hB_k^r(\R^n)$ and $\o\in \B_k^r(\R^n)$ is supported in $\supp(J)^c$, then $\o(J)=0$.  
		\item Fix a closed set $X\subseteq \R^n$.  If $J\in \hB_k^r(\R^n)$ and $\o(J)=0$ for all $\o\in \B_k^r(\R^n)$ supported in $X^c$, then $J$ is supported by $X$.
		\item \label{smallsupp}If $T$ is a continuous linear operator on $\hB_k^r(\R^n)$ with $\supp(TA)\subseteq \supp(A)$ for all $A\in \P_k(\R^n)$, then $\supp(TJ)\subseteq \supp(J)$ for all $J\in \hB_k^r(\R^n)$.
	\end{enumerate}
\end{lem}
The proof can be found in \cite{harrison8}.

\begin{defn}
	Let $\P_k(W)$ denote the space of pointed chains supported in $W$, and let $\hB_k^r(W)$ denote the space of differential chains supported in $W$.  We say $\D_U^j(p;\a)$ is \textbf{in $W$} if $\{p+ \sum a_i u_i, u_i\in U, 0\leq a_i\leq 1\}\subset W$.  In other words, we do not want to be able to ``step outside'' our open set to measure distance.  If $A\in \P_k(W)$, define
	\[
	\|A\|_{B^r,W}:=\inf\left\{\sum_{i=1}^N |\D_{U_i}^{j_i}(p_i;\a_i)|_{j_i}: A=\sum_{i=1}^N \D_{U_i}^{j_i}(p_i;\a_i),\,\,\, \D_{U_i}^{j_i}(p_i;\a_i) \textrm{ in W},\,\,\, 0\leq j_i\leq r\right\}.
	\]
\end{defn}

It is immediate that $\|\cdot\|_{B^r,W}$ is a seminorm on $\P_k(W)$ and that $\|A\|_{B^r}\leq \|A\|_{B^r,W}$.  Hence, $\|\cdot\|_{B^r,W}$ is a norm on $\hB_k^r(W)$.  Note that if $W\neq \R^n$, then $\hB_k^r(W)$ is not complete, since the space $\hB_k^r(W)$ includes Cauchy sequences of simple pointed chains converging in $\hB_k^r(\R^n)$ to a simple pointed chain supported on $\p W$.  However, we can say the following:
\begin{lem}
	Suppose $\{J_i\}_i$ is a Cauchy sequence in $\hB_k^r(W)$, and $\supp(J_i)\subset W-\bar{N}$, where $N$ is a neighborhood of $\p W$.  Then there exists $J\in \hB_k^r(W)$ such that $J_i\rightarrow J$ in $\hB_k^r(W)$ and $\supp(J)\subset W-\bar{N}$.
\end{lem}
\begin{proof}
	This follows from the fact that $\|\cdot\|_{B^r}\leq \|\cdot\|_{B^r,W}$ and from Lemma \ref{supportfacts}. 
\end{proof}

We may define the dual space to $\hB_k^r(W)$ as follows:
\begin{defn}
	If $\o\in \P_k(W)^*$, define 
	\[
	\|\o\|_{B^r,W}:=\sup\left\{\left|\cint_{\D_U^j(p;\a)}\o\right| : \D_U^j(p;\a) \textrm{ in W}, |\D_U^j(p;\a)|_j=1, 0\leq j\leq r\right\}.
	\]
	Let $\B_k^r(W):=\{\o\in \P_k(W)^* : \|\o\|_{B^r,W}<\i\}$.  
\end{defn}

One can show that $\B_k^r(W)$ is the set of such $\o\in \P_k(W)^*$ such that $f\cdot \o\in \B_k^r(\R^n)$ for all $C^\i$ functions $f$ on $\R^n$ supported in $W$ (set $f\cdot \o\equiv 0$ on $\R^n\setminus W$).  One can also show as before that $\B_k^r(W)\simeq (\hB_k^r(W))'$, and that $\|\cdot\|_{Op}=\|\cdot\|_{B^r,W}$.  Moreover, using Lemma \ref{supportfacts}.\ref{smallsupp}, we can extend the definitions of $E_X$, $E_X^\dagger$, $P_X$ and $m_f$ to $\hB_k^r(W)$.  There is a fifth fundamental operator on $\hB_k^r(W)$, and that is \emph{pushforward}.  

\subsubsection{Pushforward}\label{push}
Let $U\subseteq\R^n$ and $W\subseteq \R^m$ be open.  For $r\geq 2$, let $\mathcal{M}_\B^r(U, W)$ denote the set of differentiable functions $F: U \rightarrow W$ whose coordinate functions' directional derivatives are elements of $\B_0^{r-1}(U)$. 

\begin{defn}
	Let $A=\sum (p_i;\a_i)\in \P_k(U)$ and let $F\in \mathcal{M}_\B^r(U,W)$, $r\geq 2$.  Define
	\[
	F_* A=\sum(F(p_i), F_* \a_i).
	\]
	It is easy to see that $F_*$ is linear.  We call this map \textbf{pushforward}.
\end{defn}

\begin{lem}
	The linear map $F_*: \P_k(U)\rightarrow \P_k(W)$ is continuous in the $s$-norm topology for all $0\leq s\leq r$.  Thus, $F_*$ extends to a map $F_*: \hB_k^s(U)\rightarrow \hB_k^s(W)$ for all $0\leq s\leq r$.  
\end{lem}

\begin{proof}
	If $f: U\rightarrow W$, $m=1$, $g: W\rightarrow \R$, and $v$ is a unit vector in $\R^n$, then it follows from the chain rule that
	\[
	\frac{\|g\circ f\|_{B^1}}{\|g\|_{B^1}}\leq \max\{1, |f|_1\},
	\]
	and for $r\geq 2$,
	\[
	\frac{\|g\circ f\|_{B^r}}{\|g\|_{B^r}}\leq \max\{1, r\|D_v f\|_{B^{r-1}}\}.
	\]
	With these inequalities in mind, by breaking the function $F\in \mathcal{M}_\B^r(U,W)$ into its coordinate functions, we conclude that $\sup_{0\neq \o\in \B_k^s(W)} \|F^*\o\|_{B^s}/\|\o\|_{B^s}<\i$.  It follows that $F^*$ is continuous, and hence so is $F_*$.  
\end{proof}

\begin{lem}
	The dual operator to pushforward is \emph{pullback}.  That is,
	\[
	\cint_{F_J}\o=\cint_{J}F^*\o
	\]
	for all $J\in \hB_k^r(U)$ and all $\o\in \B_k^r(W)$.
\end{lem}

The following naturality lemma is useful for defining the operators on a manifold:
\begin{lem}\label{naturality}
	If $X\in V_\B^r(U)$, and $F\in \mathcal{M}_\B^r(U,W)$ is a diffeomorphism onto its image (required so that we may pushforward a vector field), then
	\begin{enumerate}
		\item $F_* m_{f\circ F}=m_f F_*$ for all $f\in \B_0^r(W)$,
		\item $F_* E_X=E_{F_*X}F_*$,
		\item $F_* E_X^\dagger=E_{F_*X}^\dagger F_*$,
		\item $F_* P_X=P_{F_* X}F_*$.
	\end{enumerate}
\end{lem}
\begin{proof}
	It is enough to verify these equalities on pointed chains, or alternatively using the dual operators on forms.
\end{proof}

\subsection{Differential Chains on Manifolds}
We now have a theory on open subsets of Euclidian space.  However, in general, we would like to be able to work with chains in an abstract manifold.
\begin{defn}
	Let $M$ be a smooth, complete\footnote{We assume this here only for the sake of simplifying the exposition.} Riemannian manifold and let $\P_k(M)$ denote the space of finitely supported sections of $\L^k T^*M$.    
\end{defn}

It is tempting to use \emph{sheaves} and local coordinates to define differential chains on $M$, however there is no guarantee that the \emph{restriction} of a differential chain to an open set be defined.  For example, the restriction of the pointed chains in Example \ref{der} to the open unit ball about the point $(1,0)$ diverges in all norms.  Instead, we define norms on $\P_k(M)$ directly using a modified version of the difference operator.  We will then use partitions of unity and local charts to define the operators from the previous section in this more general setting.

Let $(p;\a)\in \P_k(M)$, and let $u\in T_p M$.  If $\gamma_u(t)$ is the geodesic flow from $p$ tangent to $u$, the Levi-Civita connection $\nabla$ on $TM$ allows us to parallel-transport $\a$ along $\gamma_u(t)$ to get a $k$-vector $\a_u$ of equal mass in $\L^k T_{\exp_p u}M$.

\begin{defn}
	Define $\D_{u_1}(p;\a):=(\exp_p u_1; \a_{u_1})-(p;\a)\in \P_k(M)$.  
	
	Now, suppose $u_2\in T_p M$.  Let $u_2'\in T_{\exp_p u_1}M$ denote the parallel transport of $u_2$ on the geodesic $\gamma_{u_1}(t)$ to the point $\exp_p u_1$.  Using this notation, we make the following definition:
	\[\D_{(u_1,u_2)}^2(p;\a):=(\exp_{\exp_pu_1}u_2'; {\a_{u_1}}_{u_2'})-(\exp_p u_1; \a_{u_1})-(\exp_p u_2; \a_{u_2})+(p;\a).\]
\end{defn}

\begin{figure}[htbp]
	\centering
		\includegraphics[height=3in]{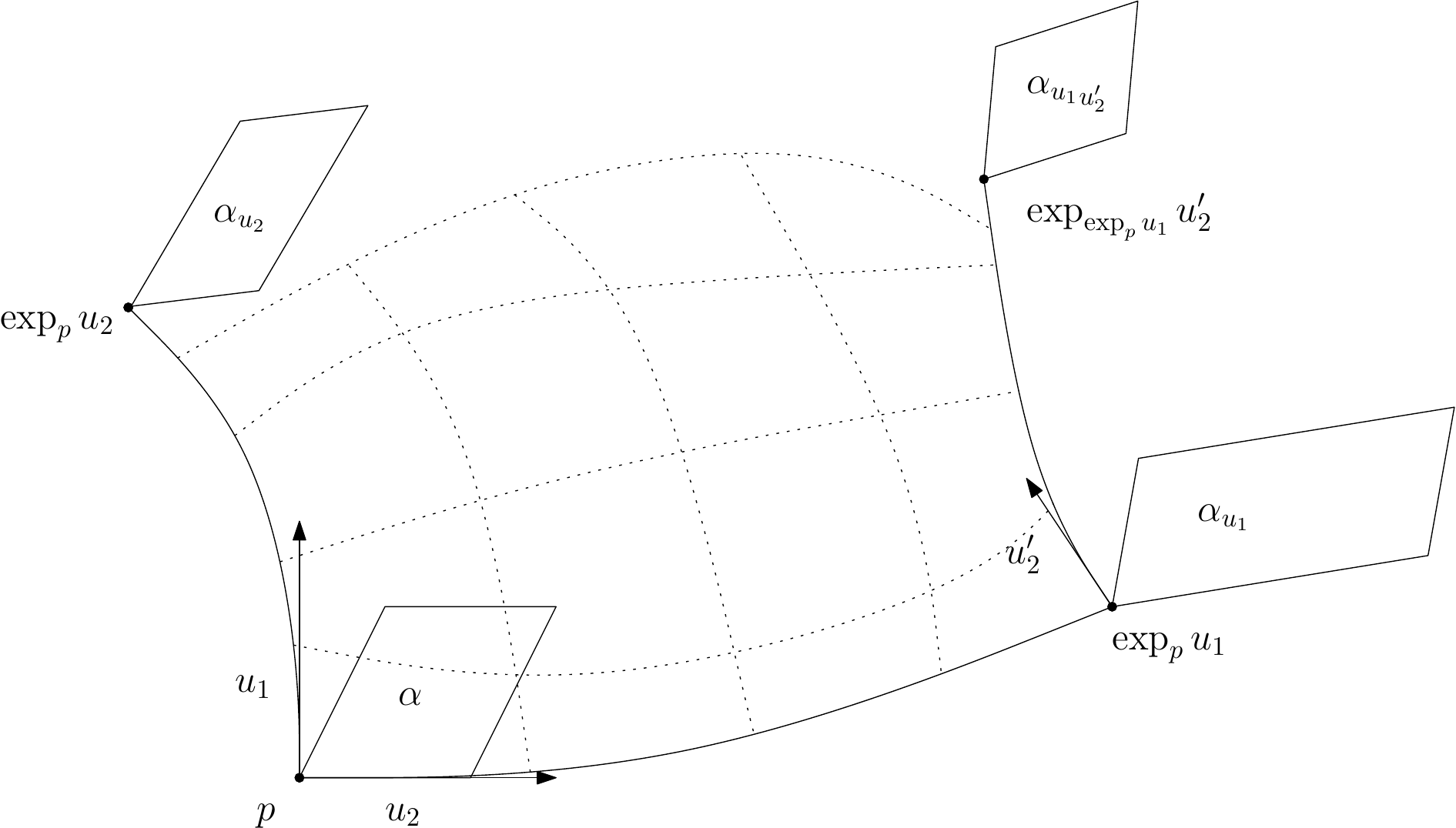}
	\caption{The difference chain $\D_{(u_1,u_2)}^2(p;\a)$}
	\label{fig:curvydifferencechain}
\end{figure}

In general, if we are given an \textbf{ordered} list of vectors $U=(u_1,\dots, u_j)$, where $u_i\in T_p M$, $1\leq i \leq j$,  we may form the difference chain $\D_U^j(p;\a)$ by parallel-transporting $\a$ and the vectors $u_2,\dots, u_j$ along $\gamma_{u_1}(t)$ the point $\exp_p u_1$, and then repeating the process at $p$ and at the new point.  See Figure \ref{fig:curvydifferencechain} for an example.  As before, we set $|\D_U^j(p;\a)|:=\|u_1\|\cdots\|u_j\|\|\a\|$.

\begin{defn}
	For each $r\geq 0$, the \textbf{$r$-norm} $\|\cdot\|_{B^r}$ on $\P_k(M)$ is given by
	\[
	\|A\|_{B^r,W}:=\inf\left\{\sum_{i=1}^N |\D_{U_i}^{j_i}(p_i;\a_i)|_{j_i} : A=\sum_{i=1}^N \D_{U_i}^{j_i}(p_i;\a_i)\right\},
	\]
where $U_i$ is an \emph{ordered} list of vectors $(u_1,\dots,u_{j_i})$ in $T_{p_i}M$, and the infimum is taken over all possible ways of writing $A$ in such a way.
\end{defn}

\begin{lem}
	The $r$-norm $\|\cdot\|_{B^r}$ on $\P_k(M)$ is in fact a norm.  
\end{lem}
\begin{proof}
	Positive homogeneity follows from the fact that $\nabla$ is the Levi-Civita connection, and hence masses are preserved by parallel transport.  Subadditivity follows from taking infimums, as before.  Positive definiteness requires only a slight modification of the proof of Theorem \ref{norm}: while Lemma \ref{voodoo} carries over with no modification necessary, we need to find $X\in \P_k(M)^*$ such that $X(A)\neq 0$ and $\max\{|X|_0,\dots,|X|_r\}<\i$.  It is enough to find some smooth $k$-form $X$ with compact support such that
	\[
	X(p_i;\a_i)=\begin{cases} 1 \text{ if $i=1$},\\
							  0 \text{ if $i\neq 1$}.
				\end{cases}
	\]
	It follows from differentiability that $\max\{|X|_0,\dots,|X|_r\}<\i$.
\end{proof}

We complete $\P_k(M)$ to get Banach space $\hB_k^r(M)$ for each $r\geq 0$.  We may define the dual space to $\hB_k^r(M)$ as follows:

\begin{defn}
	If $\o\in \P_k(M)^*$, define 
	\[
	\|\o\|_{B^r}:=\sup\left\{\left|\cint_{\D_U^j(p;\a)}\o\right| : |\D_U^j(p;\a)|_j=1, 0\leq j\leq r\right\}.
	\]
	Let $\B_k^r(W):=\{\o\in \P_k(W)^* : \|\o\|_{B^r,W}<\i\}$.    
\end{defn}

As before, it holds that $\B_k^r(W)\simeq \hB_k^r(M)'$ and that $\|\o\|_{B^r}=\|\o\|_{Op}$.  Likewise, we define a continuous \emph{multiplication} operator $m_f$ on $\hB_k^r(W)$ when $f\in \B_0^r(M)$ as before.

Let $\mathcal{A}$ be a locally finite atlas on $M$ consisting of coordinate charts $(U_i, \phi_i)$ where $\phi_i\in\mathcal{M}_\B^r(U_i, \phi_i(U_i))$ and $\phi_i^{-1}\in\mathcal{M}_\B^r(\phi_i(U_i),U_i)$.  For example, an atlas consisting of bounded normal coordinate charts works.  To define the operators $E_X, E_X^\dagger$ and $P_X$ on $\hB_k^r(M)$, we use the coordinate charts:

Let $\{\chi_i\}_i$ be a partition of unity subordinate to the $U_i$'s.  Since $\mathcal{A}$ is locally finite, we may write $J=\sum_i m_{\chi_i}J$ for all $J\in \hB_k^r(M)$.  

\begin{defn}
	Let $X\in V_\B^r(M)$, $J\in \hB_k^r(M)$.  The chain $E_X J\in \hB_{k+1}^r(M)$ is defined as follows: write $J=\sum_i m_{\chi_i}J$, $J_i:=(\phi_i)_* (\chi_i J)$, $X_i:= (\phi_i)_* X|_{U_i}$.  Define
	\[
	E_X J :=\sum_i(\phi_i^{-1})_* E_{X_i}J_i.
	\]
	Similarly define $E_X^\dagger$ and $P_X$.  It follows from Lemma \ref{naturality} that these are well defined operators on $\hB_k^r(M)$.  
\end{defn}

\subsection{Representing Classical Domains as Differential Chains}\label{classical}
\subsubsection*{Open Sets}\label{opensets}
Suppose $U\subset \R^n$ is open, with $\int_U dv<\i$.  We would like to find a chain $\tilde{U}\in \hB_n^1(\R^n)$ such that
\begin{align}\label{whatwewant}
\int_U \o=\cint_{\tilde{U}}\o
\end{align}
for all $n$-forms $\o\in \B_n^1(\R^n)$.  

We begin with the simple case of the unit $n$-cube $A$ in $\R^n$.  The $k$-th order binary subdivision of $A$ is a set of $2^{nk}$ $n$-cubes, each with $n$-volume $2^{-nk}$.  For each $k\geq 0$, let $A_k=\sum_{i=1}^{2^{nk}}(p_i;2^{-nk}e_1\wedge\cdots\wedge e_n)$, where $p_i$ is the barycenter of the $i$-th $n$-cube in the subdivision.  Since $\cint_{A_k}\o$ is just the $k$-th order Riemann sum using the binary subdivision of $A$, it follows that the integral of any Riemann-integrable $n$-form $\o$ is given by
\[
\int_A \o = \lim_{k\rightarrow \i} \cint_{A_k} \o.
\]
Thus, by continuity of the integral, it is enough to show that $\{A_k\}_k$ is a Cauchy sequence in the $1$-norm.  We may split each simple element $(p_i;2^{-nk}e_1\wedge\cdots\wedge e_n)$ constituting $A_k$ into $2^n$ elements of mass $2^{-n(k+1)}$. By Definitions \ref{normdef} and \ref{seminormdef}, it follows that
\[
\|A_k-A_{k-1}\|_{B^1}\leq 2^{nk}\cdot \left(2^{-nk}\left(n\left(2^{-(k+1)}\right)^2\right)^{\frac{1}{2}}\right)=n^{1/2}\cdot 2^{-(k+1)},
\]  
where the term $2^{nk}$ is the number of points $p_i$, the term $2^{-nk}$ is the value $\|2^{-nk}e_1\wedge\cdots\wedge e_n\|$, and the term $\left(n\left(2^{-(k+1)}\right)^2\right)^{\frac{1}{2}}$ is the distance from $p_i$ to the nearest point in $A_{k-1}$.  Figure \ref{fig:openbounded} shows this estimate in the case $k=1$.  Using a telescoping sequence and the triangle inequality, it follows that
\[
\|A_i-A_j\|_{B^1}\leq \sum_{k=j+1}^{i}\|A_k-A_{k-1}\|_{B^1}=n^{1/2}\sum_{k=j+2}^{i+1}\left(\frac{1}{2}\right)^k<\e,
\]
for all $i,j>N_{\e}$, where $N_{\e}$ is determined by the rate of convergence of the geometric series $\sum_{k=0}^\i (1/2)^k=2$.  It follows that $\{A_k\}_k$ is Cauchy and hence convergent in $\hB_n^1(\R^n)$.  Let $\tilde{A}:=\lim_{k\rightarrow\infty} A_k$.  It follows that 
\[
\int_A \o=\cint_{\tilde{A}}\o
\]
for all $\o\in \B_n^1(\R^n)$.

   \begin{figure}
   	\centering
   		\subfloat[$A_0$]{\includegraphics{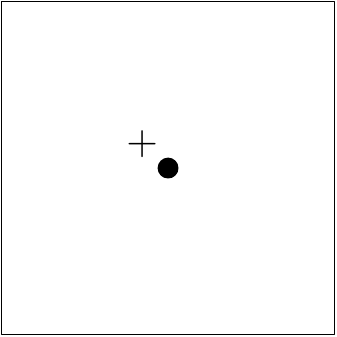}\label{fig:J_0}}\hfill
   		\subfloat[$4\cdot\frac{1}{4}A_0$]{\includegraphics{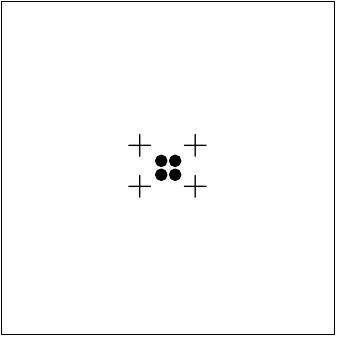}\label{fig:J_0'}}\hfill
   		\subfloat[$A_1-A_0$]{\includegraphics{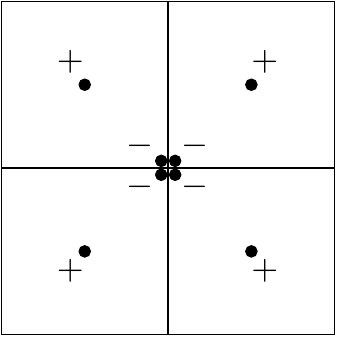}\label{fig:J_1-J_0}}\hfill
  		\subfloat[Bound on $\|A_1-A_0\|_{B^1}$]{\includegraphics{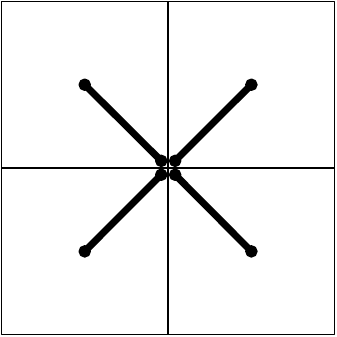}\label{fig:J_1-J_0norm}}
		\caption{Bounding $\|A_1-A_0\|_{B^1}$ by subdividing $A_0$.}\label{fig:openbounded}
   \end{figure}

More generally, we would like to do the same for any finite-volume open set $U\subset \R^n$.  This is easily achieved using a Whitney decomposition \cite[Appendix J]{grafakos} (or more generally, any decomposition of $U$ into closed $n$-cubes of diminishing size).  For each $k\geq 0$, let $J_k$ be the sum of the differential chains corresponding to cubes of size $2^{-nk}$ or larger in the Whitney decomposition.  Since $U$ is assumed to have finite volume, it follows that $\|J_i-J_j\|_{B^0}$, equal to the combined area of cubes between the sizes of $2^{-ni}$ and $2^{-nj}$, is bounded by $\e$ for $i,j$ large enough.  Thus, $\tilde{U}:=\lim_{i\rightarrow 0} J_i\in \hB_k^1(\R^n)$.  Furthermore, it is clear by construction that the integral relation (\ref{whatwewant}) is satisfied.  Note that the above construction relies on the orientation induced by $dv$.  If we reverse the orientation, then the resulting chain is simply $-\tilde{U}$.

\subsubsection*{Polyhedral Chains}
Let $\D^k$ denote the standard $k$-simplex, and let $A_k$ denote the affine subspace of $\R^{k+1}$ containing $\D^k$ as a subset.  Since open sets are canonically represented as differential chains, we may associate to the interior of $\D^k$ a $k$-chain $\widetilde{\D^k}\in \hB_k^1(A_k)$, corresponding to $\D^k$ (given an orientation) in the sense of integration.  Let $\mathcal{A}_k(\R^n)$ denote the set of all affine maps from $A_k$ to $\R^n$.  
\begin{defn}
	We define the space of \textbf{polyhedral $k$-chains in $\R^n$} to be the subspace of $\hB_k^1(\R^n)$ generated by the set
	\[
	\{\phi_* \widetilde{\D^k} : \phi\in \mathcal{A}_k(\R^n)\}.
	\] 
That is, polyhedral $k$-chains are finite sums of $k$-simplices embedded in $\R^n$.  
\end{defn}

It follows from this definition that the boundary of a polyhedral $k$-chain is a polyhedral $(k-1)$-chain.  Moreover,

\begin{lem}\label{densepolyhedral}
	The space of polyhedral $k$-chains in $\R^n$ is dense in $\hB_k^1(\R^n)$.  
\end{lem}
\begin{proof}
	By linearity, it suffices to show that we can approximate a simple $k$-chain $(p;\a)$ by a sequence of polyhedral chains.  In fact, using pushforward by an affine map, we may assume $(p;\a)=(0; e_1\wedge\cdots\wedge e_k)$.  For each $i\geq 0$, let $C_i:=\{(x_1,\dots, x_n)\in \R^n : -2^{-i-1}\leq x_j \leq 2^{-i-1} \text{ for } 0\leq j\leq k, \text{ and } x_{j+1}=\dots=x_n=0\}$.  Give $C_i$ the orientation induced by the ordering $(e_1,\dots,e_k)$.  In other words, $C_i$ is an oriented $k$-cube of side-length $2^{-i}$ centered at $0$.  As such, we may write $C_i$ as a sum of $k$-simplices, and so $C_i$ is a polyhedral $k$-chain.
	
	We show that the sequence $\{2^{ki}C_i\}_i$ converges to $(0; e_1\wedge\cdots\wedge e_k)$.  First, approximate $2^{ki}C_i$ by a sequence of pointed chains $\{Q_i^h\}_h$, each satisfying $\|Q_h^i\|_{B^0}=1$.  By rescaling, we can assume that the simple $k$-chains constituting $Q^{i+1}_h$ are translates of the simple $k$-chains constituting $Q^i_h$.  Since the $1$-norm is bounded by mass times distance translated, we may conclude that for $j\geq i$,
	
	\[
	\|2^{ki}C_i - 2^{kj}C_j\|_{B^1}= \lim_{h\rightarrow \i} \|Q^i_h- Q^j_h\|_{B^1}\leq k^{1/2}2^{-i}.
	\]
	Thus, the sequence $\{2^{ki}C_i\}_i$ converges in $\hB_k^1(\R^n)$.  Furthermore, by integrating over forms $\o\in \B_k^1(\R^n)$, we conclude that $\lim_{i\rightarrow \i}2^{ki}C_i=(0; e_1\wedge\cdots\wedge e_k)$.
	\end{proof}

More generally, we may pushforward $\widetilde{\D^k}$ by an arbitrary $C^1$ map defined on $A_k$.  The subspace of $\hB_k^r(\R^n)$ generated by the image of $\widetilde{\D^k}$ under such maps is called the space of \textbf{algebraic $k$-chains in $\R^n$}.  Note that algebraic chains are not the same type of object as singular chains, since pushforward by a map that folds the $k$-simplex back onto itself will result in the $0$ algebraic chain, but will be non-zero as a singular chain.  That is, $f(\D^k)\neq f_*(\widetilde{\D^k})$.  This difference is made up once we pass to homology, but the difference on the chain level is subtle and important.  Algebraic chains are, since they allow for cancellation, much smaller as a space than singular chains, and they form sort of ``intermediate'' step between singular chains and their homology.  They are equipped with the algebraic structure of differential chains, and as such are interesting objects of study in their own right.  For example, one can construct, using algebraic chains, embedded compact orientable submanifolds of $\R^n$.

\section{The Cauchy Theorems}
Our goal in this section is to generalize the Cauchy residue theorem so that the domain of integration is not required to be a collection of paths, but rather a closed differential $1$-chain.  To do this, we begin with the Cauchy integral theorem, and work our way through to the general residue theorem.  For the most part, we follow the presentation of the Cauchy theorems found in \cite{raostetkaer}.  That is, we make use of the same dependencies and theorem ordering, however, the proofs are all new and the results are far more general.  Before we begin, however, we state the following lemma, one which we use repeatedly in this section:
\begin{lem}[Generalized Poincar\'e Lemma]\label{poincare}
	Suppose $M$ is an $n$-dimensional Riemannian manifold, and $U\subseteq M$ is open and contractible.  Then if $J\in \hB_k^r(U)$, $1\leq k\leq n-1$, $r\geq 1$ and $\p J=0$, then there exists $K\in \hB_{k+1}^{r-1}(U)$ such that $J=\p K$.  
\end{lem}
\begin{proof}
	See \cite{harrison8}.  The proof makes use of a \emph{cone construction}.  One first deforms $U$ so that it is star-shaped, then approximates $J$ with a sequence of polyhedral chains, and then creates a \emph{cone} over the polyhedral chains.  One can then show that this cone approaches a limit whose boundary is $J$.  
\end{proof}

As an immediate consequence of Lemma \ref{poincare} and the generalized Stokes' theorem, we get the following result:

\subsection{Generalized Cauchy Integral Theorem}
\begin{thm}{Generalized Cauchy Integral Theorem}
\label{Generalized Cauchy Integral Theorem}

Let $U\subset \C$ be a bounded contractible open set, let $f: U\rightarrow \C$ be a holomorphic function, and let $J$ be a (real) closed differential $1$-chain of order $r$ for any $r\geq 1$ supported in $U\subset \C\simeq\R^2$.  Then 
\[
\cint_J f(z)dz=0.
\]
\end{thm}

\begin{proof}
	
We should note that the above integral is a bit sloppy.  We are trying to integrate a complex 1-form over a real 1-chain.  To make things rigorous, we write $f=u+iv$ where $u$ and $v$ are real valued functions.  So the integral above becomes
\begin{align}
\cint_J f(z)dz &= \cint_{J}(u(x,y)+iv(x,y))d(x+iy) \\
&=  \cint_{J}\left( u(x,y) dx  -v(x,y)dy\right)+ i\cint_{J} \left(v(x,y) dx + u(x,y)dy\right).\label{decomp}
\end{align}
Since $U$ is bounded, it follows that the $1$-forms $udx$, $vdy$, $vdx$, $udy$ are elements of $\B_1^r(U)$ for all $r\geq 0$.  So, (\ref{decomp}) consists of well-defined Harrison integrals, so one may take this to be the definition of $\cint_J f(z)dz$.  Apply Lemma \ref{poincare} to $J$ to get $J=\partial K$ where $K$ is a differential $2$-chain supported in $U$.  By Lemma \ref{stokes}, we get

\begin{align*}
\cint_J f(z)dz&=\cint_{\partial K} \left(u(x,y) dx  -v(x,y)dy\right)+ i\cint_{\partial K}\left( v(x,y) dx + u(x,y)dy\right)\\
 &= \cint_K  \left(\frac{\p u}{\p y} + \frac{\p v}{\p x}\right) dydx + i \cint_K \left(\frac{\p v}{\p y} - \frac{\p u}{\p x}\right)dydx\\
&=0,
\end{align*}
where the final equality is given by the Cauchy-Riemann Equations.  
\end{proof}

Theorem \ref{Generalized Cauchy Integral Theorem} implies the classical Cauchy integral theorem, because of the natural representation of a smooth curve as a differential 1-chain.  But we can also integrate over more exotic domains such as non-rectifiable curves, Lipschitz curves, chains supported on fractals, and divergence-free vector fields, again treating these objects as differential $1$-chains.  See Section \ref{classical} for examples of such domains.   So that we may state a generalized Cauchy residue theorem, we now give a definition of winding number for differential chains.
\begin{figure}[htbp]
	\centering
		\includegraphics[height=3in]{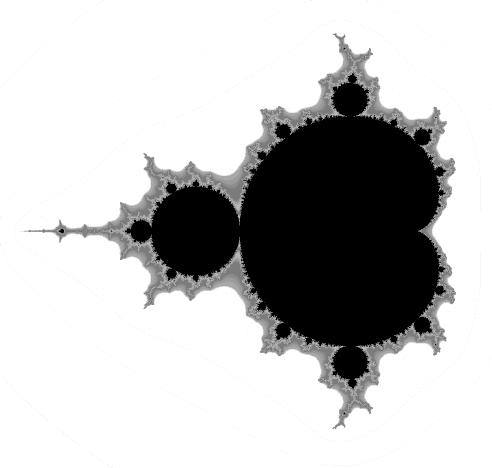}
	\caption{The boundary of the Mandelbrot set supports a closed differential $1$-chain over which we may integrate a holomorphic function.}
	\label{fig:mandelbrot}
\end{figure}

\subsection{Generalized Winding Number}
\begin{defn}[Generalized Winding Number]
\label{Generalized Winding Number}	

Let $J$ be a (real) differential $1$-chain of order $r\geq 0$ in $\C$ and let $z\in \supp(J)^c$.  Then the winding number of $J$ about $z$, $\textrm{Ind}_J(z)$ is defined to be
\[
\textrm{Ind}_J(z):=\frac{1}{2\pi i}\cint_J \frac{dw}{w-z}.
\]
\end{defn}

Note that $f(w):=\frac{1}{w-z}\in \B_0^r(U)$ for $r\geq 0$, where $U$ is any neighborhood of $\supp(J)$ whose closure does not contain $z$.  Thus, the above integral is well-defined.  Via the representations in section \ref{classical}, it follows that definition \ref{Generalized Winding Number} corresponds to the classical definition where the latter is defined.  That is, when the differential chain $J$ corresponds to a piecewise differentiable, parametrized, closed curve, the above Harrison integral is equal to its classical counterpart.  However, we need to check that our definition behaves nicely when extended to differential chains in general.  Immediately we see on connected components of $\supp(J)^c$ that $\textrm{Ind}_J(z)$ is continuous.  This follows since $\frac{1}{w-z_i}\rightarrow \frac{1}{w-z}$ in $\B_0^r(U)$ for $r\geq 0$ when $z_i\rightarrow z$, $z_i\in \supp(J)^c$, $z_i\notin \bar{U}$.  We will show further in Theorem \ref{Winding number constant on connected components} that if $J$ is closed, then the winding number is constant on connected components of $\supp(J)^c$, and in Corollary \ref{Winding number is zero on unbounded connected component} that if $J$ is closed and compactly supported, then the winding number is zero on the unbounded connected component of $\supp(J)^c$.  But first, we have an immediate result, a generalized version of the Cauchy Integral Formula:

\subsection{Generalized Cauchy Integral Formula}
\begin{thm}{Generalized Cauchy Integral Formula}
\label{Generalized Cauchy Integral Formula}

Let $U\subset \C$ be a bounded contractible open set, let $f: U\rightarrow \C$ be a holomorphic function, let $J\in \hB_1^r(\R^2)$ be supported in $U$, $r\geq 1$, and let $z\in U\setminus\supp(J)$.  Then

\[
\textrm{Ind}_J(z)f(z)=\frac{1}{2\pi i}\cint_J \frac{f(w)}{w-z}dw.
\]
\end{thm} 

\begin{proof}
	
The function
\[w\rightarrow
\begin{cases}
	\frac{f(w)-f(z)}{w-z} & \textrm{for } w\in U\setminus \{z\}\\
	f'(z)& \textrm{for } w=z
\end{cases}
\]
is holomorphic in $U$, so by the Generalized Cauchy Integral Theorem,
\[
\cint_J \frac{f(w)-f(z)}{w-z} dw =0.
\]
The theorem follows from our definition of the generalized winding number.
\end{proof}

\subsection{Properties of the Generalized Winding Number}
To show that the generalized winding number of $J$ is well-behaved, it is useful, if $J$ is closed, to approximate $J$ with a sequence of closed polyhedral chains.  Since polyhedral chains are dense in $\hB_k^r(\R^n)$ for all $r\geq 1$ (see Lemma \ref{densepolyhedral}), we know that $J$ can be approximated by polyhedral chains.  However, to insist that these polyhedral chains be closed is a strong statement that we cannot make just yet.  In fact, we will need something slightly stronger: we need the closed polyhedral chain approximation to avoid the point around which we are computing the winding number.  What follows in the next two lemmas is a proof of the existence of such a closed polyhedral chain approximation.  
 
\begin{lem}\label{topdim}
	Let $K\in \hB_m^r(S^m)$, where $S^m$ is the $m$-sphere.  If $\p K=0$, then $K=a \hat{S^m}$ for some $a\in \R$, where $\hat{S^m}\in \hB_m^r(S^m)$ denotes the chain canonically associated to $S^m$.
\end{lem}
\begin{proof}
	Note that the canonical mapping $\psi_r :\hB_m^r(S^m)\rightarrow \hB_m^\i (S^m)$ is injective (Lemma \ref{strict}).  Since $S^m$ is compact, the space $\hB_m^\i$ naturally injects into the space of de Rham currents on $S^m$ via the inclusion map $\hB_m^\i(S^m)\hookrightarrow (\B_m^\i(S^m))'$.  Boundary commutes with these maps, so we get a closed $m$-current $\hat{K}$ associated to $K$.  Since the homology of de Rham currents is dual to de Rham cohomology, we know that in particular, it satisfies the Eilenberg-Steenrod axioms \cite{eilenbergsteenrod}.  As such, the $n$-th de Rham \emph{current homology} group, denoted $H_n^{dR}(S^n)$, is isomorphic to $\R$, and so a closed $m$-current on $S^n$ is unique up to scalar.  Hence $K=a S^m$ for some $a\in \R$.     
\end{proof}

\begin{lem} 
	\label{Existance of closed polyhedral chain lemma}
	
	Let $J\in \hB_r^1(\R^2)$ be closed, $r\geq 1$, and let $z\in \supp(J)^c$.  Then there exists $0<\epsilon<1$ and a sequence of closed polyhedral chains $P_j \to J$ such that the ball of radius $\epsilon$ about $z$, $B_\epsilon(z)$ is contained in $\supp(J)^c$ and $\supp(P_j)\cap B_\epsilon=\emptyset$ for all $j$.  
\end{lem}

\begin{proof}
  Let $U$ be an open neighborhood of $\supp(J)$ such that $z$ is not contained in the closure of $U$.  So, $J\in \hB_1^r(U)$.  Let $K_j\rightarrow J$ be a sequence of polyhedral chains supported in $U$.  Let $\pi$ be the projection of $\C\setminus \{z\}$ onto the unit circle about $z$,
\[
w\mapsto \frac{w-z}{\|w-z\|}+z.
\]    
Write $K_j=\sum_{j_m} k_{j_m}$, where the $k_{j_m}$'s are individual simplices.  By splitting larger simplices into smaller ones if necessary, we may assume without loss of generality that the lengths of the $k_{j_m}$'s are bounded by $1/j$.  Let $a_{j_m}$ be the (weighted) start point of $k_{j_m}$ an let $b_{j_m}$ be the (weighted) end point of $k_{j_m}$.  That is, $\partial k_{j_m}=b_{j_m}-a_{j_m}$.  Thus, 
\[
\pi_* \partial k_{j_m}=\pi_* b_{j_m}-\pi_*a_{j_m}.
\]
We see that $a_{j_m}-\pi_* a_{j_m}$ bounds a simplex $q_{j_m}$ from $\pi_* a_{j_m}$ to $a_{j_m}$ and that $\pi_* b_{j_m}-b_{j_m}$ bounds a simplex $p_{j_m}$ from $b_{j_m}$ to $\pi_* b_{j_m}$.  Let $k'_{j_m}$ be the simplex bounded by $\pi_*a_{j_m}-\pi_* b_{j_m}$.  We can make $\e$ small enough so that the closure of $B_\e(z)$ does not intersect the closure of $U$, and that for $j$ large enough the simplices $k'_{j_m}$ do not intersect $B_\e(z)$.  (The only way this would fail would be if we projected a very long simplex onto the unit circle so that the end points were close to being antipodal.  The restriction on the lengths of these simplices makes such a situation impossible.)  Thus,

\[
r_{j_m}:=k_{j_m}+k'_{j_m}+q_{j_m}+p_{j_m}
\]
is a closed polyhedral chain supported in $U$.   It follows from the homotopy operator in \cite{harrison8} that $(R_j=\sum_{j_m} r_{j_m})\rightarrow (J - \pi_* J)$.  Since $\pi_* J$ is closed and an element of $\hB_1^r(S)$ where $S$ is the unit circle about $z$, we know, by Lemma \ref{topdim}, that $\pi_* J=a S$ where $a\in \R$. We may approximate $S$ by closed polyhedral chains $S_j$, for example, regular polygons, and so $aS_j\rightarrow \pi_* J$.  Let 
\[
P_j:=R_j-aS_j.
\]
Thus, $P_j$ is a closed polyhedral chain and as $j\rightarrow\i$, $P_j\rightarrow J$.  By our construction, the $P_j$ miss $B_\epsilon(z)$ for all $j>N$ for some $N$.
\end{proof}

\begin{thm}
	\label{Winding number constant on connected components}
	
If  $J$ is a closed, compactly supported chain in $\hB_1^r(\R^2)$, $r\geq 1$, then the winding number $\mathrm{Ind}_J$ is constant on connected components of $\supp(J)^c$.
\end{thm}

\begin{proof}  
	
	Let $z\in \supp(J)^c$, $B_\epsilon(z)$ an open neighborhood of $z$ whose closure is disjoint from $\supp(J)$ and $P_n \to J$ a sequence of closed polyhedral chains as in Lemma \ref{Existance of closed polyhedral chain lemma}.  A closed polyhedral $1$-chain $P_n$ is just a weighted piecewise linear parameterized closed curve.  That is, there exists a piecewise linear parameterized closed curve $C_n$ and a weight $\l_n\in \R$ such that
	\[
	\cint_{P_n}\o=\l_n \int_{C_n}\o
	\]
	for all $\o\in \B_1^1(\R^2)$.  Let $z_0\in B_\epsilon(z)$.  Then,
\[
\textrm{Ind}_J(z_0) =\frac{1}{2\pi i}\cint_J \frac{dw}{w-z_0} =  \lim_{n \to \i} \frac{1}{2\pi i}\cint_{P_n} \frac{dw}{w-z_0} = \lim_{n \to \i} \l_n\textrm{Ind}_{C_n}(z_0). 
\]
Since $\supp(P_n)\cap B_\epsilon(z)=\emptyset$, we know that $B_\epsilon(z)$ lies entirely within a connected component of $\supp(P_n)^c$.  Since $C_n$ is a piecewise smooth closed parameterized curve, the properties of the classical winding number hold.  In particular, $\textrm{Ind}_{C_n}(z_0)=\textrm{Ind}_{C_n}(z)$.  So, 
\[
\lim_{n \to \i}\l_n\textrm{Ind}_{C_n}(z_0)=\lim_{n \to \i}\l_n\textrm{Ind}_{C_n}(z)=\frac{1}{2\pi i}\cint_{\lim_{n\rightarrow\i} P_n}\frac{dw}{w-z}=\frac{1}{2\pi i}\cint_{J} \frac{dw}{w-z}=\textrm{Ind}_J(z).
\]	
\end{proof}

\begin{cor}
	\label{Winding number is zero on unbounded connected component}
	
	If $J$ is a compactly supported closed chain in $\hB_1^r(\R^2)$, $r\geq 1$, and $z$ is in the unbounded connected component of $\supp(J)^c$, then $\textrm{Ind}_J(z)=0$.
\end{cor}

\begin{proof}
	
	By the coning construction used in the above lemma, $\cup_n \supp(P_n)$ is bounded, and so we may choose $z$ in the unbounded component of $\supp(J)^c$ so that $z$ is also in the unbounded component of $\supp(P_n)$ for all $n$.  As in the proof of Theorem \ref{Winding number constant on connected components}, the classical properties of winding number hold for $P_n$.  Thus, 
	\[
	\textrm{Ind}_J(z)=\lim_{n\to\i}\textrm{Ind}_{P_n}(z)=0.
	\]
\end{proof}

We now know that our winding number behaves as it should.  However, we can say even more:    

\begin{defn}
	\label{density}
	Let $K\in \hB_n^r(\R^n)$ have finite mass.  The \emph{signed density} of $K$ at the point $x$ is defined to be the value
	\[
	\lim_{\epsilon\rightarrow 0}\frac{1}{\textrm{vol}(B_\epsilon)}\cint_{K_{\lfloor B_\epsilon}}dv,
	\]
	where $K_{\lfloor B_\epsilon}$ denotes the \emph{restriction} of $K$ to the $n$-ball of radius $\epsilon$ about $x$ and $\textrm{vol}(B_\epsilon)$ is the volume of the $n$-ball.  The requirement that $K$ have finite mass is a technical condition to ensure that the restriction $K_{\lfloor B_\epsilon}$ is well defined. The \textbf{mass} of a differential chain $K$ is defined to be
	\[
	\inf\{\liminf \|A_i\|_{B^0} : A_i\rightarrow K\},
	\]
	where the infimum is taken over all ways to describe $K$ as a limit of pointed chains $A_i$.  Note that mass is lower semicontinuous.
\end{defn}

\begin{thm}
	\label{Winding number is equal to density}
	Let $J$ be a closed element of $\hB_1^r(\R^2)$, $r\geq 1$.  If $K$ is a $2$-chain of finite mass with $\partial K=J$ and $z\in \supp(J)^c$, then $\textrm{Ind}_J(z)$ is equal to the signed density of $K$ at the point $z$.	
\end{thm}

\begin{proof}
	Let $J$ be a closed $1$-chain and let $z\in \supp(J)^c$ and set $K$ equal to the $2$-chain constructed via the Poincar\'e lemma by coning at the point $z$.  I.e., $\partial K=J$ and $K_n=\sum_{n_j}k_{n_j}\rightarrow K$, where $k_{n_j}$ are $2$-simplices with basepoint $z$.  For each $\epsilon>0$ let $B_\epsilon$ be the open ball of radius $\epsilon$ about $z$.  Let $m_{n_j}$ be the signed density of $k_{n_j}$, let $\theta_{n_j}$ be the angle subtended by $k_{n_j}$ at the point $z$ and let $l_{n_j}$ be the partial boundary of ${k_{n_j}}_{\lfloor B_\epsilon}$ opposite $z$.  If $K$ has finite mass, then we may conclude that 
	\[
	\lim_{n\rightarrow \i}\sum_{n_j}l_{n_j}=\partial \left(K_{\lfloor B_\epsilon}\right).
	\]
	Since $\frac{1}{w-z}$ is holomorphic on a neighborhood of $\supp(K-K_{\lfloor B_\epsilon})$, it follows from Theorem 
	\ref{Generalized Cauchy Integral Theorem} that 
\begin{align*}
	\textrm{Ind}_J(z)&=\frac{1}{2\pi i}\cint_{\partial (K_{\lfloor B_\epsilon})}\frac{dw}{w-z}\\
	&=\lim_{n\rightarrow\i}\sum_{n_j}\frac{1}{2\pi i}\cint_{l_{n_j}}\frac{dw}{w-z}\\
	&=\lim_{n\rightarrow\i}\sum_{n_j}\frac{m_{n_j}\theta_{n_j}}{2\pi},
\end{align*}
where the last integral is computed classically.  Likewise, the signed density of $K$ at $z$ is given by
\begin{align*}
	\lim_{\epsilon\rightarrow 0}\frac{1}{\pi\epsilon^2}\cint_{K_{\lfloor B_\epsilon}} dx\,dy=\lim_{\epsilon\rightarrow 0}\lim_{n\rightarrow\i}\sum_{n_j}\frac{1}{\pi\epsilon^2}\cint_{{k_{n_j}}_{\lfloor B_\epsilon}}dx\,dy=\lim_{n\rightarrow\i}\sum_{n_j}\frac{m_{n_j}\theta_{n_j}}{2\pi},
\end{align*}
where the last integral is computed classically.  Now suppose $\partial K= \partial K'=J$.  Then by the generalized Stokes' theorem, we conclude that
\[
\cint_K d\omega=\cint_{K'} d\omega
\]
for all $\o\in \B_1^r(\R^2)$.  Since all real-valued $2$-forms on $\R^2$ are exact, it follows that $K=K'$, so our choice of $K$ is unique.  
\end{proof}

\begin{lem}
	\label{Winding number zero implies poincare lemma for non-contractible sets}
	
	Let $U\subset \C$ be bounded and open and suppose $J\in \hB_1^r(\R^2)$ is supported in $U$ and closed.  Suppose there exists some $K\in \hB_2^{r-1}(\R^2)$ with finite mass such that $\p K=J$.  If $\textrm{Ind}_J(w)=0$ for all $w\in U^c$, then $K$ is supported in $U$.  
\end{lem} 

\begin{proof}
	
	The signed density of $K$ is zero outside $U$, hence $K$ is supported in $U$, whereby the lemma follows from Theorem \ref{Winding number is equal to density}. 
\end{proof}

\subsection{Generalized Global Cauchy Integral Theorem}
\begin{thm}{Generalized Global Cauchy Integral Theorem}
	\label{Generalized Global Cauchy Integral Theorem}
	
	Let $J\in \hB_1^r(\R^2)$, $r\geq 1$ be closed and supported in a bounded open set $U\subset \C$ such that $\textrm{Ind}_J(w)=0$ for all $w\in U^c$.  Suppose there exists some $K\in \hB_2^{r-1}(\R^2)$ with finite mass such that $\p K=J$. Then if $f$ is holomorphic on $U$,
	\[
	\cint_J f(z)dz=0.
	\]
\end{thm}

\begin{proof}
	
	By Lemma \ref{Winding number zero implies poincare lemma for non-contractible sets}, there exists some $K$ supported in $U$ such that $J=\partial K$.  The proof is otherwise identical to that of Theorem \ref{Generalized Cauchy Integral Theorem}.
\end{proof}

\subsection{Generalized Global Cauchy Integral Formula}
\begin{thm}{Generalized Global Cauchy Integral Formula}
	\label{Generalized Global Cauchy Integral Formula}
	
	Let $J\in \hB_1^r(\R^2)$, $r\geq 1$ be closed and supported in a bounded open set $U\subset \C$ such that $\textrm{Ind}_J(w)=0$ for all $w\in U^c$.  Suppose there exists some $K\in \hB_2^{r-1}(\R^2)$ with finite mass such that $\p K=J$.  Then if $f$ is holomorphic on $U$ and $z\in U\setminus \supp(J)$,
	\[
	f(z)\textrm{Ind}_J(z)=\frac{1}{2\pi i}\cint_J \frac{f(w)}{w-z}dw.
	\]
\end{thm}

\begin{proof}
	
	This follows from Theorem \ref{Generalized Global Cauchy Integral Theorem} in the same manner as Theorem \ref{Generalized Cauchy Integral Formula} followed from Theorem \ref{Generalized Cauchy Integral Theorem}.
\end{proof}

\subsection{Generalized Cauchy Residue Theorem}
\begin{thm}{Generalized Cauchy Residue Theorem}
	\label{Generalized Residue Theorem}
	
	Let $J\in \hB_1^r(\R^2)$ be closed and supported in a bounded open set $U\subset \C$ such that $\textrm{Ind}_J(w)=0$ for all $w\in U^c$.  Suppose there exists some $K\in \hB_2^{r-1}(\R^2)$ with finite mass such that $\p K=J$.  Let $f$ be holomorphic in $U$ except for at finitely many points $a_k\in U\setminus\supp(J)$.  Then,
	\[
	\cint_J f(z)dz=\sum_k \textrm{Ind}_J(a_k)\cint_{B_k} f(z)dz,
	\]
	where $B_k=\p D_k$ and the $D_k\ni a_k$ are isolated open neighborhoods.
\end{thm}

\begin{proof}
	
	Since $\textrm{Ind}_z(J)=0$ for all $z\in U^c$, it follows from Lemma \ref{Winding number zero implies poincare lemma for non-contractible sets} that there exists $K$ supported in $U$ such that $\partial K= J$.  For each $k$, let $D_k$ be an open ball around $a_k$ such that the closures of the $D_k$'s are disjoint from each other and contained in $U\setminus \supp(J)$.  Let $B_k=\p D_k$.  Let $D_k'$ be the $2$-chain in $\R^2$ that corresponds canonically to $D_k$ and let $B_k'=\p D_k'$.  
	
	By Theorems \ref{Winding number constant on connected components} and \ref{Winding number is equal to density}, the signed density of $K$ is constant on connected components of $U\setminus \supp(J)$.  It follows that the signed density of $K$ at the point $a_k$ is equal to the signed density of $K$ on any point in $D_k$.  Thus, $\supp(K-\sum_k \textrm{Ind}_J(a_k)D_k')=\supp(K)\setminus(\cup_k D_k)$.  Therefore, $f$ is holomorphic on a neighborhood of $\supp{K-\sum_k \textrm{Ind}_J(a_k)D_k'}$.  By Theorem \ref{Generalized Global Cauchy Integral Theorem},
	\[
	\cint_{\partial(K-\sum_k \textrm{Ind}_J(a_k)D_k')}f(z)dz=0.
	\] 
	Therefore, 
	\[
	\cint_J f(z)dz=\cint_{\partial K}f(z)dz=\cint_{\partial(\sum_k \textrm{Ind}_J(a_k)D_k')}f(z)dz=\sum_k \textrm{Ind}_J(a_k)\cint_{B_k'}f(z)dz.
	\]
\end{proof}
    
\begin{remark}
	We conjecture that the requirement that $J$ be closed in Theorem \ref{Generalized Residue Theorem} is unnecessary.  The requirement that $\textrm{Ind}_J(w)=0$ for all $w\in U^c$ is very strong and should imply that $\p J=0$.  Moreover, the requirement that $K$ have finite mass could also probably be dropped.  We only need it so that signed density is well defined.  If we modify the definition of signed density to use multiplication by a smooth bump function rather than the restriction of $K$ to an $\e$-ball, the theorem might hold.
\end{remark}

\section{Asymptotic Cycles}
\subsection{The Asymptotic Cycles of Schwartzman}
The long-term behavior of dynamical systems is a subject of great interest.  In $\R^2$, the situation is relatively simple.  Poincar\'e-Bendixson proved \cite{poincare,bendixon} that a bounded orbit of a continuous dynamical system on the plane approaches a periodic orbit.  In higher dimensions, and on some $2$-dimensional surfaces, more complicated situations arise.  For example, the orbits of an irrational flow on the torus do not approach a periodic orbit.  However, these long-term orbits are, in a certain sense, cycles.  

The idea of an \emph{asymptotic cycle} was first introduced by Schwartzman in \cite{schwartzman}.  Given a compact metric space $X$ with finite first Betti number and a continuous flow $\phi: X\times \R\rightarrow X$, Schwartzman associates to each \emph{quasi-regular} point $p\in X$ an element $A_p$ of the first \u{C}ech homology group of $X$.  Schwartzman proves the following theorem as a geometric interpretation:
\begin{thm}[Schwartzman]
	Let $p$ be any quasi-regular point.  For each $t$ let $K_t$ be some parameterized curve going from $\phi(p,t)$ to $p$.  Suppose that all the curves $K_t$ are parametrized uniformly equicontinuously from the interval $[0,1]$.  Let $C_t$ be the curve determined by the orbit from $p$ to $\phi(p,t)$ followed by the curve $K_t$, and let $\bar{C}_t$ be the corresponding element of the first Betti group.  Then $\lim_{t\rightarrow\i}\bar{C}_t/t=A_p$.
\end{thm}

Loosely speaking, the asymptotic cycle tells us how the orbit from a point $p$ wraps around the space $X$ in the long run.  They are highly related to winding numbers, foliations, and Hamiltonian flows.  In particular, Schwartzman deduces some facts about Hamiltonian systems:

If $(X,\o)$ is a compact symplectic $C^2$ manifold, we may write $\o$ in local coordinates as $\o=\sum dp_i\wedge dq_i$.  If $\a$ is a closed $C^1$ 1-form on $X$, then we can find a function $H(p_i, q_i)$ such that, locally, $\a=dH$.  One then gets a system of differential equations
\begin{align*}
	\frac{d q_i}{dt}=\frac{\p H}{\p p_i},\,\,\,\,&\,\,\,\frac{d p_i}{dt}=-\frac{\p H}{\p q_i}
\end{align*}
determined by $\a$.  One can then deduce the existence of a flow $D_\a$, called a \emph{Hamiltonian flow} on $X$ determined by these differential equations.  Schwartzman proves the following result:
\begin{thm}
	The $\mu$-asymptotic cycle $A_\mu$ associated with $D_\a$ is completely determined by the de Rham cohomology class of $\a$.  Furthermore, the resulting map sending cohomology classes to $\mu$-asymptotic cycles is linear.
\end{thm}

Schwartzman then notes that one can use this result to compute winding number of a Hamiltonian flow on the torus.  

\subsection{Foliation Cycles}
More recently, the ideas of Schwartzman have been reframed in the language of foliations and currents by Sullivan, Ruelle, Plante and others in \cite{ruellesullivan,sullivan,schwartzman2}.  See \cite{candelconlon,candelconlon2} for a general reference.  Generally speaking, the dynamical system is replaced with a \emph{foliation}.  For a 1-dimensional foliation, the terms \emph{foliation cycle} and \emph{asymptotic cycle} are synonymous.  In higher dimensions, foliation currents are a powerful tool in determining the structure of a given foliation.  In particular, Sullivan uses foliation currents to show that a $(n-1)$-dimensional foliation is \emph{taut} if and only if the foliation cycles are not exact.  This condition is equivalent to the leaves being minimal with respect to a Riemannian metric on the manifold.

\subsection{Asymptotic Cycles and Differential Chains}
	The space $\hB_k^\i(M)$ is isomorphic to a subspace of de Rham currents.  Note that the spaces $\hB_k^r(M)$ are \emph{not} reflexive, for $0\leq r\leq\i$.  Our goal is to construct asymptotic differential chains in the space $\hB_1^1(M)$, thus making available integration over a broader class of forms than with de Rham currents.	
	
Let $X$ be a Lipschitz vector field on a compact Riemannian manifold $(M,g)$.  Let $p\in M$ and let $\phi_p(t)$ be the the time-$t$ flow of $X$ from $p$.  The point $p$ is called \textbf{quasi-regular} if
\[
\lim_{T\rightarrow \infty}\frac{1}{T}\int_0^T f(\phi_p(t))dt
\]
exists for every real-valued continuous function $f$ on $M$.  It is well known that the set of quasi-regular points has full measure with respect to every measure $\mu$ on $M$ invariant under the flow of $X$ \cite{schwartzman}. 

Let $\widetilde{[0,T]}\in \hB_1^1(\R)$ denote the differential chain canonically associated to the interval $[0,T]$ (see Section \ref{classical}.)  Since $M$ is compact, it follows that $\phi_p\in \mathcal{M}_\B^1(\R, M)$, and so the pushforward operator ${\phi_p}_*$ is defined and continuous on $\hB_1^1(\R)$.  Therefore, we are able to define the differential chain
\[
J_T:=\frac{1}{T}{\phi_p}_*\widetilde{[0,T]}.
\]
Our goal is to show $\lim_{T\rightarrow\infty}J_T$ is a differential chain.  By Theorem \ref{equivnorm}, we know that
\[
\|J_t-J_s\|_{B^1}=\sup_{\o\in \B_1^1(M), \|\o\|_{B^1}=1}\left|\cint_{J_t-J_s}\o\right|.
\]
By weak-* compactness of the unit ball in $\B_1^1(M)$, and by weak-* continuity of $\cint$, it follows that the supremum is achieved by some form $\o_{t,s}\in \B_1^1(M)$ with $\|\o_{t,s}\|_{B^1}=1$, depending on $s$ and $t$.   So, we may write
\begin{align*}
\|J_t-J_s\|_{B^1}&=\left|\cint_{J_t}\o_{t,s}-\cint_{J_s}\o_{t,s}\right|\\
&=\left|\frac{1}{t}\cint_{\widetilde{[0,t]}}{\phi_p}^*(\o_{t,s})-\frac{1}{s}\cint_{\widetilde{[0,s]}}{\phi_p}^*(\o_{t,s})\right|\\
&=\left|\frac{1}{t}\int_0^t{\phi_p}^*(\o_{t,s})-\frac{1}{s}\int_0^t{\phi_p}^*(\o_{t,s})\right|.
\end{align*}
The form ${\phi_p}^*(\o_{t,s})$ applied to $(t;\l)\in T\R$ gives
\[
{\phi_p}^*(\o_{t,s})(t;\l)=\l\o_{t,s}({\phi_p}(t);X({\phi_p}(t))).
\]
On the other hand, the form ${\phi_p}^*(\iota_X\o_{t,s})dt$ applied to $(t;\l)$ gives
\[
{\phi_p}^*(\iota_X\o_{t,s})dt(t;\l)=\l(\iota_X\o_{t,s})({\phi_p}(t))=\l \o_{t,s}({\phi_p}(t);X({\phi_p}(t))).
\]
Thus, ${\phi_p}^*(\o_{t,s})={\phi_p}^*(\iota_X\o_{t,s})dt$, and we have
\[
\|J_t-J_s\|_{B^1}=\left|\frac{1}{t}\int_0^t{\phi_p}^*(\iota_X\o_{t,s})dt-\frac{1}{s}\int_0^s{\phi_p}^*(\iota_X\o_{t,s})dt\right|.
\]
Let $\Phi_t(f):=\frac{1}{t}\int_0^t {\phi_p}^*(f)dt$.  Likewise, let $\Phi(f):=\lim_{t\rightarrow\i}\Phi_t(f)$.  If $f_i\rightarrow f$ uniformly, it follows that ${\phi_p}^*f_i\rightarrow {\phi_p}^* f$ uniformly, which implies that $\Phi_t$ is a continuous linear operator on $\mathcal{C}(M,\R)$ where $\mathcal{C}(M,\R)$ is the space of continuous functions on $M$, and given the topology of uniform convergence.  Since $|\Phi_t(f)|\leq \sup\{|f(q)| : q\in M\}$, we likewise know that $\Phi(f)$ is a continuous linear operator on $\mathcal{C}(M,\R)$.  Thus, $\Phi_t\rightarrow \Phi$ in the weak-* topology on $(\mathcal{C}(M,\R))'$.  

Suppose the functions $\iota_X\o_{t,s}$ belong to a compact subset $\mathcal{F}\subset\mathcal{C}(M,\R)$.  Since the family $\{\Phi_t\}_t$ is equicontinuous, given $\e>0$ there exists a finite collection of functions $\{f_i\}_{i\in I}\subset\mathcal{F}$ such that if $g\in \mathcal{F}$, then there is some $i\in I$, such that the inequality $|\Phi_t(f_i)-\Phi_t(g)|<\e$ is satisfied for all $t$.  Furthermore, since $\Phi_t$ a priori converges uniformly to $\Phi$ on finite subsets, it follows that there exists $N>0$, independent of $g$, such that if $s,t>N$, then
\begin{align*}
\left|\Phi_t(g)-\Phi_s(g)\right|\leq \left|\Phi_t(g)-\Phi_t(f_i)\right|+\left|\Phi_t(f_i)-\Phi_s(f_i)\right|+\left|\Phi_s(f_i)-\Phi_s(g)\right|\leq 3\e.
\end{align*}
This implies in particular that 
\[
\|J_t-J_s\|_{B^1}=\left|\Phi_t(\iota_X\o_{t,s})-\Phi_s(\iota_X\o_{t,s})\right|<3\e.
\] 
By the Arzel\`a-Ascoli theorem, if $\mathcal{F}=\{f\in \mathcal{C}(M,\R) : \lip(f)<C, \sup\{f(q): q\in M\}<C \}$ for some constant $C>0$, then $\mathcal{F}$ is compact.  So, to show $\{J_t\}_t$ is Cauchy, it suffices to show:
\begin{lem}\label{fancypower}
	There exists some constant $C>0$ such that if $\o\in \B_1^1(M)$, $\|\o\|_{B^1}=1$, then $\lip(\iota_X\o)<C$ and $\iota_X\o(q)<C$ for all $q\in M$.
\end{lem}
\begin{proof}
	Since $\iota_X$ is the dual operator to $E_X$, which is continuous on $\hB_1^1(M)$, we know that $\iota_X$ is continuous on $\B_1^1(M)$.  Hence, 
	\[
	\sup_{\o\in \B_1^1(M), \|\o\|_{B^1}=1} \|\iota_X \o\|_{B^1}<\i.
	\]
	By definition of $\|\cdot\|_{B^1}$, it follows that setting $C=\sup_{\o\in \B_1^1(M), \|\o\|_{B^1}=1} \|\iota_X \o\|_{B^1}$ suffices.
\end{proof}

Thus, we have proved the following:

\begin{thm}
	Let $M$ be a compact Riemannian manifold and let $X$ be a Lipschitz vector field on $M$.  If $p\in M$ is quasi-regular, then $J_p:=\lim_{T\rightarrow \i}J_T\in \hB_1^1(M)$. 
\end{thm}

Since $\p$ commutes with ${\phi_p}_*$ (just as exterior derivative commutes with pullback), it follows that 
\[
\p J_p=\p \lim_{n\rightarrow\i}{\phi_p}_*\frac{1}{t}\widetilde{[0,t]}=\lim_{n\rightarrow\i}{\phi_p}_*\frac{1}{t}\p \widetilde{[0,t]}=\lim_{t\rightarrow\i}({\phi_p}(t);1/t)-({\phi_p}(0);1/t).
\]
Since $M$ is compact, we know that $d({\phi_p}(t),{\phi_p}(0))<R$, where $R<\i$.  Therefore, $\|({\phi_p}(t);1/t)-({\phi_p}(0);1/t)\|_{B^1}\leq R/t\rightarrow 0$.  Thus, $\p J_p=0$, and we call $J_p$ the \textbf{asymptotic differential cycle associated to $p$}.

Given a Radon measure $\mu$ on $M$ and a $1$-form $\o$ on $M$, we may form the integral
\begin{align}\label{current}
\int_M \iota_X\o d\mu.
\end{align}
We would like to find a differential chain $\xi_{X,\mu}\in \hB_1^1(M)$ such that  
\begin{align}\label{whatwewant2}
\cint_{\xi_{X,\mu}}\o=\int_M \iota_x\o d\mu
\end{align}
for all $\o\in \B_1^1(M)$.  Since convex convex combinations of Dirac measures are weak-$*$ dense in the set of probability measures, by scaling, we may find a sequence of measures $\mu_i$ converging weak-$*$ to $\mu$ with
\[
\mu_i=\sum_{k_i=1}^{n_i}c_{k_i}\d_{p_{k_i}},
\]
where $\sum_{k_i=1}^{n_i}c_{k_i}=\mu(M)$ and $\d_{p_{k_i}}$ is the Dirac measure at the point $p_{k_i}$.  Set $K_i=\sum_{k_i=1}^{n_i}(p_{k_i}; c_{k_i}X(p_{k_i}))$.  Then
\[
\o(K_i)=\cint_{K_i}\o=\int_M \iota_X\o d\mu_i\rightarrow \int_M \iota_X \o d\mu,
\]
for all $\o\in \B_1^1(M)$.  So, $\{K_i\}_i$ is weakly Cauchy.  We show that in fact $\{K_i\}_i$ is strongly Cauchy.  Since the set $\{\mu_i\}_i$ is equicontinuous and weak-$*$ convergent to $\mu$, where the measures are considered as elements of $(\mathcal{C}(M,\R), \|\cdot\|_{sup})'$, it follows that $\mu_i\rightarrow \mu$ uniformly on compact subsets of $\mathcal{C}(M,\R)$.  By Lemma \ref{fancypower}, it follows that given $\e>0$ there exists $N$ such that if $s,t>N$ then
\[
\|K_t-K_s\|_{B^1}=\left| \cint_{J_t}\o_{t,s}-\cint_{J_s}\o_{t,s}\right|=\left|\mu_t(\iota_X\o_{t,s})-\mu_s(\iota_X\o_{t,s})\right|<\e.  
\]
Therefore, $\{K_i\}_i$ is in fact strongly Cauchy, and therefore convergent to some element in $\hB_1^1(M)$.  Call this element $\xi_{X,\mu}$.  It follows from weak convergence that $\xi_{X,\mu}$ satisfies (\ref{whatwewant2}).  So, we have proved:
\begin{thm}
	Let $M$ be a compact Riemannian manifold and let $X$ be a Lipschitz vector field on $M$.  If $\mu$ is a Radon measure on $M$, then there exists a chain $\xi_{X,\mu}\in \hB_1^1(M)$ such that
	\[
	\cint_{\xi_{X,\mu}}\o=\int_M \iota_X\o d\mu
	\]
	for all $\o\in \B_1^1(M)$.
\end{thm}

\begin{thm}
	The measure $\mu$ is invariant under the flow of $X$ if and only if $\p\xi_{X,\mu}=0$.
\end{thm}
\begin{proof}
	For all $f\in \B_0^2(M)$, we have
	\[
	\cint_{\p\xi_{X,\mu}}f=\cint_{\xi_{X,\mu}}df=\int_M X(f)d\mu=\lim_{t\rightarrow 0}\frac{1}{t}\left[\int_M f(\phi_x(t))d\mu-\int_M f(x)d\mu\right].
	\]
	Since $\B_0^2(M)$ is dense in $\mathcal{C}(M,\R)$, this value is $0$ if and only if $\mu$ is invariant under the flow.
\end{proof}

This gives a general notion of ``divergence free,'' since the Lebesgue measure $\mu_L$ is invariant under the flow of $X$ if and only if $X$ is divergence free if and only if $\p\xi_{X,\mu_L}=0$.

Suppose now that $\mu$ is ergodic with respect to the flow of $X$.  Then for almost all $p\in M$,
\[
\lim_{T\rightarrow\i}\frac{1}{T}\int_0^T \phi_p^* f dt=\frac{1}{\mu(M)}\int_M fd\mu
\]
for all $f\in \mathcal{C}(M,\R)$.  In particular, this is true for all $\iota_x\o$, and hence we have the following theorem:
\begin{thm}
	If $\mu$ is ergodic with respect to the flow of $X$, then for almost all $p\in M$,
	\[
	J_p=\frac{1}{\mu(M)}\xi_{X,\mu}.
	\]
\end{thm}

	From here, one can determine the \emph{differential homology} class of $J_p$, as well as $\xi_{X,\mu}$ when the flow of $X$ preserves $\mu$.  We hope to generalize to higher dimensions: given a collection $X_1,\dots, X_n$ of Lipschitz vector fields on $M$, we should be able to construct a corresponding differential chain that would behave similarly to a foliation cycle.  We also hope to construct an ``average differential chain.''  Schwartzman constructs an ``average asymptotic cycle,'' $A_\mu$, that is the average of the $A_p$ over $p\in M$ with respect to $\mu$.  It would be interesting to do the same thing here, as well as in the higher dimensional case.

\section{Appendix}
\subsection{Proof of Theorem \ref{norm}}\label{normproof}

We know at this point that $\|\cdot\|_{B^r}$ is a semi-norm on $\P_k$.  It remains to show positive-definiteness.  

Let $X\in \P_k^*$, where $\P_k^*$ is the algebraic dual of $\P_k$.  Let $|X|_j=\sup\{X(\D_U^j(p;\a)) : |\D_U^j(p;\a)|_j=1\}\in[0,\i]$.  By way of the definition of $|\D_U^j(p;\a)|_j$, it follows that the set over which we are taking the supremum is non-empty.  Thus, the value $|X|_j$ is well-defined. 
\begin{lem}\label{voodoo}
	The inequality $|X(A)|\leq \max\{|X|_0,\dots,|X|_r\}\cdot\|A\|_{B^r}$ holds for all $X\in \P_k^*$, all $r\geq 0$ such that $\max\{|X|_0,\dots,|X|_r\}<\i$ and all $A\in \P_k$.  
\end{lem}
\begin{proof}
	Fix $A$ and $X$ and let $\e>0$.  Then we can write $A=\sum_{i=1}^N\D_{U_i}^{j_i}(p_i;\a_i)$ (in the sense of definition \ref{normdef}) such that
	\[
	\|A\|_{B^r}>\sum_{i=1}^N|\D_{U_i}^{j_i}(p_i;\a_i)|_{j_i}-\e.
	\]
	Then
	\begin{align}
		|X(A)|&\leq \sum_{i=1}^N |X(\D_{U_i}^{j_i}(p_i;\a_i))|\\
		&\leq \sum_{i=1}^N |X|_{j_i} |\D_{U_i}^{j_i}(p_i;\a_i)|_{j_i}\\
		&\leq \max\{|X|_0,\dots,|X|_r\}\sum_{i=1}^N |\D_{U_i}^{j_i}(p_i;\a_i)|_{j_i}\\
		&< \max\{|X|_0,\dots,|X|_r\}(\|A\|_{B^r}+\e).
	\end{align}
	Since this inequality holds for all $\e>0$, the result follows.
\end{proof}

Suppose $A\neq 0$ where $A\in\P_k$.  It suffices to find $X\in \P_k^*$ such that $X(A)\neq 0$ and $\max\{|X|_0,\dots,|X|_r\}<\i$, for this will imply by Lemma \ref{voodoo} that $\|A\|_{B^r}>0$.

Write $A=\sum_{i=1}^N(p_i;\a_i)$.  The support of $A$ is the set $\{p_1,\dots,p_N\}\subset \R^n$.  Without loss of generality, we may assume $A(p_1)=e_I$ for some multi-index $I$.  Choose a smooth function $\phi :\R^n \rightarrow \R$ with compact support such that $\phi(p_1)=1$, and $\phi(p_i)=0$, for $2\leq i\leq N$.  Define
\begin{align}
X(p; e_J)\begin{cases}\phi(p),&J=I,\\
				0,&J\neq I,
			\end{cases}
\end{align}
where $I$ and $J$ are multi-indices.  Extend to $\P_k$ by linearity.  It follows by construction that $X(A)=1$.  Since $\phi$ has compact support, we know for all $1\leq r <\i$ that 
\[
\sup |D^r \phi (p)|  <\i,
\] 
where the supremum is taken over all points $p\in \R^n$ and all $r$-th order directional derivatives $D^r$.  Call this value $|\phi|_r$.  

We aim to show that $|X|_j<\i$ for each $j\geq 0$.  Clearly this is true for $j=0$.  Thus, it suffices to show that \begin{align}\label{needtoshow}
	|X(\D_U^j(p;\a))|\leq \|u_1\|\dots\|u_j\||\phi|_j,
\end{align} for all $\|\a\|=1$ and $j\geq 1$.  We show something slightly more general, for we will need it in a lemma later on:

\begin{lem}\label{generalstuff}
	Suppose $f\in C^j(\R^n)$, $j\geq 1$.  Then $f(\D_U^j(p;1))\leq \|u_1\|\cdots\|u_j\||f|_j$ for all points $p$ and all $j$-length collections of vectors $U$.  Here, $f(\D_U^j(p;1))$ simply means apply the function $f$ at the vertices of the parallelepiped generated by $\D_U^j(p;1)$, taking orientation into account.
\end{lem}
\begin{proof}
Write $f(\D_U^j(p;1))=\sum (-1)^i f(q_i)$, where the $q_i$ are the alternating vertices (allowing repetition) of the (perhaps degenerate) parallelepiped determined by $U$ and $p$.  So, we need to show
\begin{align}\label{grrr}
\left|\sum (-1)^i f(q_i)\right|\leq \|u_1\|\cdots\|u_j\||f|_j.
\end{align}
This follows from the fundamental theorem of calculus and induction: suppose $j=1$ and $f\in C^1(\R^n)$.  Then $\D_{u_1}(p;\a)=(p+u_1;\a)-(p;\a)$,  and we set $q_2=p+u_1$, $q_1=p$.  So,
\[
\left|f(q_2)-f(q_1)\right|=\left|\int_{Q_1} \frac{\p}{\p v_1}f dv_1\right|,
\]
where $Q_1=\{p+tu_1: 0<t<1\}$ and $v_1=u_1/\|u_1\|$.  Since $\frac{\p}{\p v_1}f$ is bounded by $|f|_1$, we have
\[
\left|\int_{Q_1} \frac{\p}{\p v_1}f dv_1\right|\leq \left|\int_{Q_1} dv_1\right||f|_1=\|u_1\||f|_1.
\]

Now suppose (\ref{grrr}) holds for all functions $g\in C^j(\R^n)$ when $j<k$.  We show that if $f\in C^k(\R^n)$, then $f(\D_U^k(p;1))\leq \|u_1\|\cdots\|u_k\||f|_k$.  Recursively define $Q_h=(T_{u_h})_* Q_{h-1}-Q_{h-1}$.  That is, the $Q_h$ are formal sums of translates of the line segment $Q_1$.  By the fundamental theorem of calculus applied to a formal sum of intervals,
\begin{align*}
	\left|\sum (-1)^i f(q_i)\right|&=\left|\int_{Q_k} \frac{\p}{\p v_1}f dv_1\right|\\
	&=\left|\int_{(T_{u_k})_*Q_{k-1}}\frac{\p}{\p v_1}f dv_1-\int_{Q_{k-1}}\frac{\p}{\p v_1}f dv_1\right|\\
	&=\left|\int_{Q_{k-1}}T_{u_k}^*\frac{\p}{\p v_1}f dv_1-\int_{Q_{k-1}}\frac{\p}{\p v_1}f dv_1\right|\\
	&=\left|\int_{Q_{k-1}}(T_{u_k}^*-Id)\frac{\p}{\p v_1}f dv_1\right|\\
	&=\left|\int_{Q_{k-1}}\frac{\p}{\p v_1} (T_{u_k}^*-Id)f dv_1\right|,
\end{align*}  
where $I$ is the identity map.  Since $(T_{u_k}^*-Id)f\in C^{k-1}(\R^n)$, it follows by the inductive step that 
\[
\left|\int_{Q_{k-1}}\frac{\p}{\p v_1} (T_{u_k}^*-Id)f dv_1\right|\leq \|u_1\|\cdots\|u_{k-1}\||(T_{u_k}^*-Id)f|_{k-1}\leq \|u_1\|\cdots\|u_{k}\||f|_{k},
\]
where the last inequality is given by the mean value theorem.
\end{proof}

Since $\phi\in C^\i(\R^n)$, the inequality (\ref{needtoshow}) follows from Lemma \ref{generalstuff} and since $X(\D_U^j(p;\a))=\phi(\D_U^j(p;\a))$.

\subsection{Proof of Theorem \ref{junk}}\label{junkproof}
\begin{enumerate}
	\item For $r\geq 0$, the function $\|\cdot\|_{B^r}: \B_k^r\rightarrow \R_+$ is a norm, and turns $\B_k^r$ into a Banach space.
		\begin{proof}
			We prove only that $\|\cdot\|_{B^r}$ is a norm.  That $\B_k^r$ is complete will follow from \#\ref{topequiv1}.
			
			Let $\o\in \B_k^r$ and $a\in \R$.  Since
			\begin{align}
			|a\o|{B^j}&=\sup\{|a\o(\D_U^j(p;\a))| : |\D_U^j(p;\a)|_j=1\}\\
			&=|a|\sup\{|\o(\D_U^j(p;\a))| : |\D_U^j(p;\a)|_j=1\}=|a||\o|_{B^j},
			\end{align}
			it follows that
			\begin{align}
			\|a\o\|_{B^r}=\max\{|a\o|_{B^0},\dots,|a\o|_{B^r}\}=|a|\|\o\|_{B^r}.
			\end{align}
			Now, suppose $\o=0$.  Then clearly $\|\o\|_{B^r}=0$.  Likewise, if $\|\o\|_{B^r}=0$, then 
			\[
			\sup\{|\o(\D_U^j(p;\a))| : |\D_U^j(p;\a)|_j=1\}=0 
			\]
			for all $0\leq j\leq r$.  In particular,
			\[
			\sup \{|\o(p;\a) : |(p;\a)|_0=\|\a\|=1\}=0.
			\]
			It follows that $\o=0$.
			
			Finally, suppose $\o, \k\in \B_k^r$.  Then for each $0\leq j\leq r$, we have
			\begin{align}
			|\o+\k|_{B^j}&=\sup\{|(\o+\k)(\D_U^j(p;\a))| : |\D_U^j(p;\a)|_j=1\}\\
			&\leq \sup\{|\o(\D_U^j(p;\a))|+|\k(\D_U^j(p;\a))| : |\D_U^j(p;\a)|_j=1\}\\
			&\leq\sup\{|\o(\D_U^j(p;\a))| : |\D_U^j(p;\a)|_j=1\}+\sup\{|\k(\D_U^j(p;\a))| : |\D_U^j(p;\a)|_j=1\}\\
			&=|\o|_{B^j}+|\k|_{B^j}.
			\end{align}
			So, 
			\begin{align}
				\|\o+\k\|_{B^r}&=\max\{|\o+\k|_{B^0},\dots,|\o+\k|_{B^r}\}\\
				&\leq \max\{|\o|_{B^0}+|\k|_{B^0},\dots,|\o|_{B^r}+|\k|_{B^r}\}\\
				&\leq \max\{|\o|_{B^0},\dots,|\o|_{B^r}\}+\max\{|\k|_{B^0},\dots,|\k|_{B^r}\}\\
				&=\|\o\|_{B^r}+\|\k\|_{B^r}.
			\end{align}
		\end{proof}
		
	\item For $s\geq r\geq 0$, we have natural and continuous inclusions $\iota_{s,r}: \B_k^{s}\rightarrow \B_k^{r}$.
		\begin{proof}
			Since $\|\cdot\|_{B^r}\leq \|\cdot\|_{B^{r+1}}$, the inclusion $\B_k^{r+1}\hookrightarrow \B_k^r$ is immediate, as is its continuity.  The inclusions are natural, since each space $\B_k^r$ is defined as a subspace of all bounded $k$-forms.  
		\end{proof}
		
	\item The space $\B_k^0$ is equal to the space of bounded $k$-forms.
		\begin{proof}
			Since $\|\o\|_{B^0}=\sup\{|\o(p;\a): \|\a\|=1\}$, the result follows.
		\end{proof}

	\item The space $\B_k^1\subset \B_k^0$ is the subspace of such forms that are Lipschitz continuous.
	
	\begin{proof}
		Recall that a differential $k$-form $\o$ on $\R^n$ is called \emph{Lipschitz} if there exists a constant $C$ such that
		\[
		|\o(p+v;\a)-\o(p;\a)|\leq C \|v\|
		\]
		for all $\a$ such that $\|\a\|=1$.  This is equivalent to requiring
		\[
		|\o(p+v;\b)-\o(p;\b)|\leq C
		\]
		for all $\b$ such that $\|v\|\|\b\|=1$.  However, this condition is equivalent to the requirement that $|\o|_{B^1}< \i$.
	\end{proof}

	\item For $r> 1$, it holds that $\o\in\B_k^r$ if and only if $\o\in C^{r-1}$, its $j$-th directional derivatives are bounded for $0\leq j\leq r-1$, and its $(r-1)$-th directional derivatives are Lipschitz continuous.
	\begin{proof}
		This will follow from Theorem \ref{topequiv2}.  We do not use this fact in the rest of the proof of Theorem \ref{junk}.
	\end{proof}

	\item For $r\geq 0$, the space $\B_k^r$ is naturally isomorphic to $(\hB_k^r)'$, the continuous dual of $\hB_k^r$.  Furthermore, the injection map $\iota_{s,r}$ is the transpose of $\phi_{r,s}$. 
	\begin{proof}
		Let $X\in (\hB_k^r)'$.  Then $\|X\|_{\mathrm{Op}}^r=\sup \{ |X(A) : \|A\|_{B^r}=1\}$.  Define the differential form $\Psi(X)(p;\a):=X(p;\a)$.  That is, on the left hand side, we evaluate $\Psi(X)$ at the point $p$ and on the element $\a\in \L^k T_p \R^n$, and on the right hand side, we evaluate $X$ on the pointed chain $(p;\a)$.  We show that $\|\Psi(X)\|_{B^r}<\i$.  By definition, 
		\[
		\|\Psi(X)\|_{B^r}=\max \{|\Psi(X)|_{B^0},\dots,|\Psi(X)|_{B^r}\}.
		\]
		So, we are required to show that $|\Psi(X)|_{B^j}<\i$ for $0\leq j\leq r$.  Unravelling the definitions, we get
		\begin{align}
			|\Psi(X)|_{B^j}&=\sup\{|X(\D_U^j(p;\a))| : |\D_U^j(p;\a)|_j\leq 1\}\label{line1}\\
			&\leq \sup \{|X(\D_U^j(p;\a))| : \|\D_U^j(p;\a)\|_{B^r}\leq 1\}\label{line2}\\
			&\leq \sup \{|X(A)| : \|A\|_{B^r}\leq 1\}=\|X\|_{\mathrm{Op}}^r<\i,
		\end{align}
		where (\ref{line2}) follows from (\ref{line1}), since $\|\D_U^j(p;\a)\|_{B^r}\leq |\D_U^j(p;\a)|_j$ (definition \ref{normdef}.)
		
		So, $\Psi : (\hB_k^r)'\rightarrow \B_k^r$.  It is also, as one can check, linear.  To show that it is an isomorphism, we find an inverse:
		
		Let $\o\in\B_k^r$.  Define $\Theta: \B_k^r\rightarrow (\hB_k^r)'=(\P_k, \|\cdot\|_{B^r})'$ by the following: if $A\in \P_k$ where $A=\sum_{i=1}^N(p_i;\a_i)$, then set
		\[
		(\Theta \o)A := \sum_{i=1}^N \o (p_i;\a_i).
		\]
		By the linearity of $\o_{p_i}$, it follows that this map is a well-defined element of $\P_k^*$.  We show that $\|\Theta\o\|_{\mathrm{Op}}^r=\sup\{|\Theta\o(A)|: \|A\|_{B^r}=1\}<\i$.  By Lemma \ref{voodoo}, we know that
		\[
		|\Theta\o(A)|\leq \max\{|\Theta\o|_{B^0},\dots,|\Theta\o|_{B^r}\}\|A\|_{B^r}.
		\]
		It follows that
		\[
		\|\Theta\o\|_{\mathrm{Op}}^r\leq \max\{|\Theta\o|{B^0},\dots,|\Theta\o|_{B^r}\}=\|\o\|_{B^r}<\infty.
		\]
		To see that $\Theta$ is an inverse of $\Psi$, consider
		\begin{align*}
			(\Psi\circ\Theta)\o(p;\a)&=(\Theta\o)(p;\a)=\o(p;\a)\\
			(\Theta\circ\Psi)X(A)&=\sum \Psi(X)(p_i;\a_i)=\sum X(p_i;\a_i)=X(A).
		\end{align*}
		It follows that $\Psi$ is indeed an invertible linear map and that $(\hB_k^r)'$ and $\B_k^r$ are isomorphic.  To see that the injection map $\iota_{s,r}$ is the transpose of $\phi_{r,s}$, we use the fact that $(\hB_k^r)'=(\P_k, \|\cdot\|_{B^r})'\subset \P_k^*$.  Since $\phi_{r,s}\equiv \mathrm{Id}$ on $\P_k$, it follows that if $\o\in \P_k^*$ then $\phi_{r,s}^t\o=\o$.  In particular, this is true for any $\o\in (\hB_k^r)'$.  It follows that $\phi_{r,s}^t\o=\iota_{s,r}\o$.
	\end{proof}

	\item For $r\geq 0$ and for all $\o\in \B_k^r$, the norm $\|\o\|_{B^r}$ is equal to $\|\o\|_{\mathrm{Op}}^r$, where $\|\cdot\|_{\mathrm{Op}}^r$ is the operator norm on $\B_k^r$ considered as the dual space to $\hB_k^r$.
	\begin{proof}
		By definition, we have 
		\[
		\|\o\|_{B^r}=\max\{|\o|_{B^0},\dots, |\o|_{B^r}\}
		\]
		and
		\[
		\|\o\|_{\mathrm{Op}}^r=\sup\{|X(A): \|A\|_{B^r}=1, A\in \hB_k^r\}.
		\]
		As in (\ref{line1}), we have
		\[
		|\o|^j=\sup\{|\o(\D_U^j(p;\a))| : |\D_U^j (p;\a)|_j\leq 1\} \leq \|\o\|_{\mathrm{Op}}^r.
		\]
		Thus, $\|\o\|_{B^r}\leq \|\o\|_{\mathrm{Op}}^r$.  To see the reverse inequality, note that if $A\in \hB_k^r$, with $\|A\|_{B^r}=1$ and if $A_i\rightarrow A$ where $A_i\in \P_k$, $A_i\neq 0$, then $A_i/\|A_i\|_{B^r}\rightarrow A$ as well.  This can be seen by the triangle inequality.  In particular, this means that $\P_k$ is dense on the unit ball, and that
		\[
		\|\o\|_{\mathrm{Op}}^r=\sup\{|X(A)|: \|A\|_{B^r}=1, A\in \P_k^r\}.
		\]
		By Lemma \ref{voodoo}, it follows that
		\[
		\|\o\|_{\mathrm{Op}}^r\leq \|\o\|_{B^r}.
		\]
		From this we conclude that $\|\cdot\|_{B^r}$ turns $\B_k^r$ into a Banach space, since the continuous dual of a Banach space is again complete in the operator norm topology.
	\end{proof}

	\item The projective limit $\B_k^\i=\lim_{\leftarrow}\B_k^r$ via the inclusion mappings in \#\ref{inc} is naturally isomorphic to $(\hB_k^\i)'$.  The space $\B_k^\i$ is a Fr\'{e}chet space.
	
		This follows from some more general facts about inductive limits of topological vector spaces:
		\begin{lem}\label{general}
			Let $(X_i, \phi_{i,j})$ be a directed system of topological vector spaces over a field $\F$ such that for each $i\geq 0$ there exist dense subspaces $Y_i$ of $X_i$ such that $\phi_{i,j}(Y_i)=Y_j$ for all $j\geq i\geq 0$.  Then the continuous dual spaces $X_i'$ together with the transpose maps $\phi_{i,j}^t$ form a projective system of topological vector spaces, and the maps $\phi_{i,j}^t$ are injective.
		\end{lem}
			\begin{proof}
				It is clear that the system forms a projective limit.  It remains to show that the maps are injective.  Suppose $f,g\in X_j'$ and $\phi_{i,j}^t(f)=\phi_{i,j}^t(g)$.  Then $\phi_{i,j}^t(f)(x)=(\phi_{i,j}^* f)(x)=f(\phi_{i,j}(x))=g(\phi_{i,j}(x))$ for all $x\in X_i$.  In particular, this is true for all $x\in Y_i$.  Since $\phi_{i,j}(Y_i)=Y_j$, it follows that $f$ and $g$ agree on a dense subspace of $X_j$, and therefore are equal.  Thus, the maps $\phi_{i,j}^t$ are injective.
			\end{proof}
			
		\begin{lem}\label{general2}
			Let $(X_i, \phi_{i,j})$ be a directed system of topological vector spaces over a field $\F$.  Then $\lim_\leftarrow X_i'$ is naturally isomorphic as a vector space to $(\lim_\rightarrow X_i)'$.
		\end{lem}
			\begin{proof}
				Let $f\in (\lim_{\rightarrow}X_i)'$.  Then $\psi_i^* f\in X_i'$ for each $i\in \N$.  Furthermore, we know that $\psi_s\circ\phi_{r,s}=\psi_r$ for all $r\in \N$ and $s\geq r$.  So, $\phi_{r,s}^t \psi_s^*f =\psi_r^* f$.  It follows that $(\psi_i^* f)\in \lim_\leftarrow X_i'$.  Define the map 
				\begin{align*}
				\Phi: (\lim_{\rightarrow}X_i)'&\rightarrow  \lim_\leftarrow X_i'\\
				f&\mapsto (\psi_i^* f).
				\end{align*}
				It is clear from this definition that $\Phi$ is linear.  We show that $\Phi$ is injective.  Suppose $f, g \in (\lim_{\rightarrow}X_i)'$ with $f\neq g$.  Then $f([\a])\neq g([\a])$ for some $\a\in X_r$, $r\geq 0$.  But this means that $\psi_r^* f(\a)\neq \psi_r^* g(\a)$, hence $\Phi(f)\neq \Phi(g)$.
							
				To see that $\Phi$ is surjective, let $\xi=(f_i)\in \lim_\leftarrow X_i'$. By the universal property of the inductive limit, there is a unique continuous linear map $\bar{\xi}\in (\lim_\rightarrow X_i)'$ that makes the following diagram commute:
				\[ 
				\xymatrix{ 
				X_i \ar@/_/[rddd]_{f_i}\ar[rd]^{\psi_i}\ar[rr]^{\phi_{i,j}} & & X_j \ar[ld]_{\psi_j}\ar@/^/[lddd]^{f_j} \\
				&\lim_\rightarrow X_i \ar@{.>}[dd]^{\bar{\xi}} \\
				&\\
				&\F
				} 
				\]
				
				Furthermore, since $\psi_i^* \bar{\xi}=f_i$, we conclude that $\Phi(\bar{\xi})=(f_i)=\xi$.  In other words, $\Phi$ is surjective.
			\end{proof}
		
		Thus, from Theorem \ref{junk}.\ref{trans} and Lemma \ref{general} and \ref{general2}, we conclude that $\B_k^\i$ is isomorphic as a vector space to $(\hB_k^\i)'$.  Of course, by the definition of the spaces $\B_k^r$, it is clear that the morphisms are injective.  However, injectivity also follows from the above more general fact (Lemma \ref{general}) about topological vector spaces.  That the space $\B_k^\i$ is a Fr\'{e}chet space can be found in \cite{thomas}.
					
\end{enumerate}

\subsection{Proof of Theorem \ref{topequiv2}}\label{topequiv2proof}
We need the following extension of Lemma \ref{generalstuff}:
\begin{lem}\label{generalstuff2}
If if $\o$ is a $C^j$ $k$-form on $\R^n$, $j\geq 1$, and if $\a$ is a simple $k$-vector, $\|\a\|=1$, then
\[
|\o(\D_u^j(p;\a))|\leq \|u_1\|\cdots\|u_j\| |\o|_{C^j}.
\]
\end{lem}
\begin{proof}
	This follows from Lemma \ref{generalstuff} by writing $\o$ in terms of a coordinate function in the $k$-direction of $\a$.
\end{proof}

We want to show that $\|\o\|_{C^r}=\|\o\|_{B^r}$ for all $r\geq 0$ and all bounded $k$-forms $\o$.  It is clear from the definitions that $\|\o\|_{C^0}=\|\o\|_{B^0}$.  For $r\geq 1$, we first show $\|\o\|_{B^r}\leq \|\o\|_{C^{r-1+\lip}}$.  We may suppose $\|\o\|_{C^{r-1+\lip}}<\i$, otherwise we are done.  Therefore, $|\o|_{C^j}<\i$ for all $0\leq j\leq r-1$, and $|\o|_{L^r}<\i$.  It suffices to show for simple $k$-vectors $\a$ with $\|\a\|=1$:
\begin{enumerate}
	\item[(i)]
	$|\o(\D_U^j(p;\a))|\leq \|u_1\|\cdots\|u_j\||\o|_{C^j}$ for all $0\leq j\leq r-1$, and
	\item[(ii)]
	$|\o(\D_U^r(p;\a))|\leq \|u_1\|\cdots\|u_r\||\o|_{L^r}$.
\end{enumerate}
(i): Since $|\o|_{C^j}<\infty$, it follows by definition of $|\cdot|_{C^j}$ that $\o$ is of differentiability class $C^{r-1}$.  It follows from Lemma \ref{generalstuff2} that $|\o(\D_U^j(p;\a))|\leq \|u_1\|\cdots\|u_j\||\o|_{C^j}$.

(ii): Since $|\o|_{L^r}<\i$, it follows that the $(r-1)$-st directional derivatives of $T_{u_1}^*\o-\o$ are Lipschitz continuous.  So,
\[
\frac{|D^{r-1} (T_{u_1}^*  \o-\o)(p;\a)|}{\|u_1\|}\leq |\o|_{L^r}.
\]
It follows that $|T_{u_1}^*\o-\o|_{C^{r-1}}\leq \|u_1\| |\o|_{L^r}$.  By Lemma \ref{generalstuff2}, we have
\[
|\o(\D_U^r(p;\a))|=|(T_{u_1}^*\o-\o)(\D_{U'}^{r-1}(p;\a))|\leq \|u_2\|\cdots\|u_r\||T_{u_1}^*\o-\o|_{C^{r-1}}\leq \|u_1\|\cdots\|u_r\||\o|_{L^r}.
\]
We now show that $\|\o\|_{C^{r-1+\lip}}\leq \|\o\|_{B^r}$.  As above, we may assume $\|\o\|_{B^r}<\i$, otherwise we are done.  It is enough to show 
\begin{enumerate}
	\item[(iii)] $|D^j\o(p;\a)|\leq |\o|_{B^j}$ for all $j$-order directional derivatives $D^j$, $0\leq j\leq r-1$, $p\in \R^n$ and all $\a\in \L^k T_p \R^n$ with $\|\a\|=1$, and
	\item[(iv)] $\lip(D^{r-1}\o)\leq |\o|_{B^r}$ for all $(r-1)$-order directional derivatives $D^{r-1}$.
\end{enumerate}

First, however, we need to know that $\|\o\|_{B^r}<\i$ implies $\o$ is $(r-1)$-times differentiable.  This breaks up into a sequence of lemmas:

\begin{lem}\label{lem1}
	If $\o\in \B_k^2$, then $\o$ is differentiable.
\end{lem}
\begin{proof}
	Let $\o\in \B_k^2$.  By Lemma \ref{der3}, we know that $H_v(p;\a):=\lim_{i\rightarrow\infty}(p+2^{-i}v;2^i\a)-(p;2^i\a)\in \hB_k^2$.  By continuity of $\o$, it follows that 
	\[
	\o(H_v(p;\a))=\lim_{i\rightarrow\infty}\o(p+2^{-i}v;2^i\a)-\o(p;2^i\a).
	\]
	However, the right hand side is none other than $\lim_{h\rightarrow\o}\frac{\o(p+hv;\a)-\o(p;\a)}{h}=\mathcal{L}_v\o$.  It follows that $\o$ is differentiable.
\end{proof}

\begin{lem}\label{lem2}
	If $\o\in \B_k^r$ and $r\geq 2$, then $\|\mathcal{L}_v\o\|_{B^{r-1}}\leq \|v\|\|\o\|_{B^{r}}$.
\end{lem}

\begin{proof}
 	By Lemma \ref{lem1}, we know that $\mathcal{L}_v\o$ exists.  By Definition \ref{topequiv0}, we are required to show $|\mathcal{L}_v \o|_{B^j}\leq \|v\|\|\o\|_{B^r}$ for all $0\leq j\leq r-1$.  This inequalities follow, since
	\begin{align}
		\frac{\mathcal{L}_v\o(p;\D_U^j(p;\a))|}{\|u_1\|\cdots\|u_j\|\|\a\|}&=\lim_{t\rightarrow 0}\frac{|\o(\D_U^j(p+tv;\a/t)-\D_U^j(p;\a/t))|}{\|u_1\|\cdots\|u_j\|\a\|}\\
		&\leq \|\o\|_{B^{j+1}}\limsup_{t\rightarrow 0} \frac{\|\D_{(tv,U)}^{j+1}(p;\a/t)\|_{B^{j+1}}}{\|u_1\|\cdots\|u_j\|\|\a\|}\\
		&=\leq \|v\|\|\o\|_{B^{j+1}}\leq \|v\|\o\|_{B^r}
	\end{align}
\end{proof}

\begin{lem}\label{lem3}
	If $\o\in \B_k^r$ and $r\geq 1$, then $\|\mathcal{L}_{u_s}\circ\cdots\circ\mathcal{L}_{u_1} \o\|_{B^t}\leq \|u_1\|\cdots\|u_s\|\|\o\|_{B^{s+t}}$ for all $0\leq s+t\leq r$.
\end{lem}
\begin{proof}
	This follows from repeated applications of Lemmas \ref{lem1} and \ref{lem2}. 
\end{proof}

We conclude that $\|\o\|_{B^r}<\i$ implies $\o$ is $(r-1)$-times differentiable.

\begin{lem}\label{lastbit}
	The inequality $|D^{j-1}\o|_{B^1}\leq |\o|_{B^j}$ holds if $\o$ is $(j-1)$-times differentiable.
\end{lem}
\begin{proof}
	If $D^s$ denotes differentiation in the directions $(v_1,\dots,v_{s})$, then for $|\D_U^k(p;\a)|_k=1$ and $1\leq k\leq j-1$, it follows that
	\[
	|D^{j-k}\o(\D_U^k(p;\a))|=\lim_{t\rightarrow0}\left|(T_{tv_{j-k}}^* D^{j-k-1}\o-D^{j-k-1}\o)\left(\D_U^k\left(p;\frac{\a}{t}\right)\right)\right|\leq |D^{j-k-1}\o|_{B^{k+1}}.
	\]
	 It follows that $|D^{j-k}\o|_{B^k}\leq |D^{j-k-1}\o|_{B^{k+1}}$ for $1\leq k\leq j-1$, whence the lemma follows.
\end{proof}

The inequality (iii) follows from Lemma \ref{lastbit}: 
\[
|D^j\o(p;\a)|=\lim_{t\rightarrow 0}\left|(T_{tv_{j}}^*D^{j-1}\o-D^{j-1}\o)\left(p;\frac{\a}{t}\right)\right|\leq |D^{j-1}\o|_{B^1}\leq |\o|_{B^j}.
\]

The inequality (iv) also follows from Lemma \ref{lastbit}:
\[
\lip(D^{r-1}\o)=\sup \left|(T_v^*D^{r-1}\o-D^{r-1}\o)\left(p;\frac{\a}{\|v\|}\right)\right|\leq \sup|D^{r-1}\o(\D_v(p;\a))|=|D^{r-1}\o|_{B^1}\leq |\o|_{B^r},
\]
since $\o$ is $(r-1)$-times differentiable.

\qed

\subsection{The homomorphisms $\phi_{r,s}: \hB_k^r\rightarrow \hB_k^s$ are injective}\label{strict}
The following proof requires the \emph{prederivative operator} (Definition \ref{pre}), so in the logical sequence, this proof comes after the section on operators.  We will use this lemma only once, and that is in Lemma \ref{topdim}.  

For each $\eta>0$, let $\overline{\kappa_\eta}: \R^+\rightarrow\R$ be a smooth, monotonically decreasing function, constant on some interval $[0,t_0]$ and equal to $0$ for $t\geq \eta$.  Define $\kappa_\eta:\R^n\rightarrow \R$ by setting $\kappa_\eta(v):=\overline{\kappa_\eta(\|v\|)}$.  Re-normalize each $\kappa_\eta$ such that $\int_{\R^n}\kappa_\eta(v)dv=1$.  For $X\in \B_k^r$, define
\[
X_\eta(A):=\int_{\R^n}\kappa_\eta(v)(X(T_vA))dv
\]
where $A\in \P_k$.  It is immediate from the definition that $X_\eta\in \B_k^r$.  In fact,
\begin{lem}\label{estimates}
	If $X\in \B_k^r$ and $\eta>0$, then
	\begin{enumerate}
		\item[(i)] $\mathcal{L}_v(X_\eta)=(\mathcal{L}_vX)_\eta$ for all $v\in \R^n$,
		\item[(ii)] $\|X_\eta\|_{B^r}\leq \|X\|_{B^r}$,
		\item[(iii)] $X_\eta\in \B_k^{r+1}$,
		\item[(iv)] $X_\eta (J)\rightarrow X(J)$ as $\eta\rightarrow 0$ for all $J\in \hB_k^r$.  
	\end{enumerate}
\end{lem}

\begin{thm}\label{injectthm}
	If $J\in \hB_k^r$, then $\|J\|_{B^r}=\sup \{|X(J)| : X\in \B_k^{r+1}, \|X\|_{B^r}=1\}$.
\end{thm}
\begin{proof}
	By Theorem \ref{equivnorm}, we may write
	\[
	\|J\|_{B^r}=\sup\{|X(J)| : X\in \B_k^r, \|X\|_{B^r}=1\}.
	\]
	Let $\e>0$.  By Lemma \ref{estimates} (ii), (iii) and (iv), there exists $\eta>0$ such that if $0\neq X\in \B_k^r$ then
	\[
	\frac{|X(J)|}{\|X\|_{B^r}}<\frac{|X_\eta(J)|+\e}{\|X_\eta\|_{B^r}}\leq \sup_{0\neq Y\in \B_k^{r+1}}\frac{|Y(J)|+\e}{\|Y\|_{B^r}}.
	\]
	Since this inequality holds for all $X\neq 0$ and all $\e>0$, it follows that
	\[
	\sup_{0\neq X\in \B_k^r} \frac{|X(J)|}{\|X\|_{B^r}}\leq \sup_{0\neq Y\in \B_k^{r+1}} \frac{|Y(J)|}{\|Y\|_{B^r}}.
	\]
	On the other hand, since $\B_k^{r+1}\subset \B_k^r$, the reverse inequality holds, and so we have equality.
\end{proof}

\begin{cor}
	The homomorphisms $\phi_{r,s}: \hB_k^r\hookrightarrow \hB_k^s$ are injections for all $r\leq s$.  
\end{cor}
\begin{proof}
	By Theorem \ref{injectthm} and Theorem \ref{junk}.\ref{trans}, if $0\neq J\in \hB_k^r$, there is some $X\in \B_k^{r+1}$ such that $X(\phi_{r,r+1}(J))\neq 0$.  Thus, $\phi_{r,r+1}(J)\neq 0$.  Since $\phi_{r,s}=\phi_{s-1,s}\circ\cdots\circ\phi_{r,r+1}$, the result follows.
\end{proof}

\addcontentsline{toc}{section}{References}
\bibliography{mybib}{}

\providecommand{\bysame}{\leavevmode\hbox to3em{\hrulefill}\thinspace}
\providecommand{\MR}{\relax\ifhmode\unskip\space\fi MR }
\providecommand{\MRhref}[2]{%
  \href{http://www.ams.org/mathscinet-getitem?mr=#1}{#2}
}
\providecommand{\href}[2]{#2}
\begin{thebibliography}{Har09b}

\bibitem[AR64]{robertson}
Alex and Wendy Robertson, \emph{Topological vector spaces}, Cambridge
  University Press, Cambridge, 1964.

\bibitem[Ben01]{bendixon}
Ivar Bendixon, \emph{Sur les courbes d{\'{e}}finies par des {\'{e}}quations
  diff{\'{e}}rentielles}, Acta Mathematica \textbf{24} (1901), no.~1, 1--88.

\bibitem[Bou81]{bourbaki}
Nicolas Bourbaki, \emph{Elements of mathematics: Topological vector spaces},
  Springer-Verlag, Berlin, 1981.

\bibitem[CC99a]{candelconlon}
Alberto Candel and Lawrence Conlon, \emph{Foliations}, vol.~I, American
  Mathematical Society, Providence, RI, 1999.

\bibitem[CC99b]{candelconlon2}
\bysame, \emph{Foliations}, vol.~II, American Mathematical Society, Providence,
  RI, 1999.

\bibitem[dR73]{derham1}
Georges de~Rham, \emph{Vari{\'{e}}t{\'{e}}s diff{\'{e}}rentiables: Formes,
  courants, formes harmoniques}, Hermann, Paris, 1973.

\bibitem[dR84]{derham2}
\bysame, \emph{Differentiable manifolds}, Springer-Verlag, Berlin, 1984.

\bibitem[ES52]{eilenbergsteenrod}
Samuel Eilenberg and Norman Steenrod, \emph{Foundations of algebraic topology},
  Princeton University Press, Princeton, NJ, 1952.

\bibitem[Fed69]{federer}
Herbert Federer, \emph{Geometric measure theory}, Springer, Berlin, 1969.

\bibitem[FF60]{federerfleming}
Herbert Federer and Wendell~H. Fleming, \emph{Normal and integral currents},
  The Annals of Mathematics \textbf{72} (1960), no.~3, 458--520.

\bibitem[Gra08]{grafakos}
Loukas Grafakos, \emph{Classical fourier analysis}, Springer, Berlin, 2008.

\bibitem[Har93]{harrison1}
Jenny Harrison, \emph{Stokes' theorem on nonsmooth chains}, Bulletin of the
  American Mathematical Society \textbf{29} (1993), 235--242.

\bibitem[Har98a]{harrison3}
\bysame, \emph{Continuity of the integral as a function of the domain}, Journal
  of Geometric Analysis \textbf{8} (1998), no.~5, 769--795.

\bibitem[Har98b]{harrison2}
\bysame, \emph{Isomorphisms of differential forms and cochains}, Journal of
  Geometric Analysis \textbf{8} (1998), no.~5, 797--807.

\bibitem[Har99]{harrison9}
\bysame, \emph{Flux across nonsmooth boundaries and fractal gauss/green/stokes'
  theorems}, Journal of Physics A \textbf{32} (1999), no.~28, 5317--5327.

\bibitem[Har04a]{harrison4}
\bysame, \emph{Cartan's magic formula and soap film structures}, Journal of
  Geometric Analysis \textbf{14} (2004), no.~1, 47--61.

\bibitem[Har04b]{harrison5}
\bysame, \emph{On plateau's problem for soap films with a bound on energy},
  Journal of Geometric Analysis \textbf{14} (2004), no.~2, 319--329.

\bibitem[Har05]{harrison11}
\bysame, \emph{Lectures on chainlet geometry - new topological methods in
  geometric measure theory}, 2005.

\bibitem[Har06]{harrison6}
\bysame, \emph{Geometric hodge star operator with applications to the theorems
  of gauss and green}, Mathematical Proceedings of the Cambridge Philosophical
  Society \textbf{140} (2006), no.~1, 135--155.

\bibitem[Har07]{harrison7}
\bysame, \emph{From point cloud data to differential geometry}, Proceedings
  ICIAM (2007).

\bibitem[Har09a]{harrison8}
\bysame, \emph{Geometric poincar{\'{e}} lemma and generalizations of the
  intermediate and mean value theorems}, preprint submitted, available on
  request (2009).

\bibitem[Har09b]{harrison10}
\bysame, \emph{Personal communication}, 2009.

\bibitem[KM40]{kreinmilman}
Mark Krein and David Milman, \emph{On extreme points of regular convex sets},
  Studia Mathematica \textbf{9} (1940), 133--138.

\bibitem[K{\"{o}}t66]{kothe}
Gottfried K{\"{o}}the, \emph{Topological vector spaces}, vol.~I,
  Springer-Verlag, Berlin, 1966.

\bibitem[Mor88]{morgan}
Frank Morgan, \emph{Geometric measure theory: A beginners guide}, Academic
  Press, London, 1988.

\bibitem[Poi92]{poincare}
Henri Poincar{\'{e}}, \emph{Sur les courbes d{\'{e}}finies par une
  {\'{e}}quation diff{\'{e}}rentielle}, Oeuvres \textbf{1} (1892).

\bibitem[Roy88]{royden}
Halsey Royden, \emph{Real analysis}, Prentice-Hall, Engelwood Cliffs, NJ, 1988.

\bibitem[RS75]{ruellesullivan}
David Ruelle and Dennis Sullivan, \emph{Currents, flows and diffeomorphisms},
  Topology \textbf{14} (1975), 319--327.

\bibitem[RS91]{raostetkaer}
Murali Rao and Henrik Stetk{\ae}r, \emph{Complex analysis: an invitation},
  World Scientific, Singapore, 1991.

\bibitem[Rud91]{rudin}
Walter Rudin, \emph{Functional analysis}, McGraw-Hill, New York, NY, 1991.

\bibitem[Sch55]{schwartz2}
Laurent Schwartz, \emph{Espaces de fonctions differentiables {\`{a}} valeurs
  vectorielles}, Journal d'Analyse Mathematique \textbf{4} (1954-1955).

\bibitem[Sch57]{schwartzman}
Sol Schwartzman, \emph{Asymptotic cycles}, The Annals of Mathematics
  \textbf{66} (1957), no.~2, 270--284.

\bibitem[Sch59]{schwartz3}
Laurent Schwartz, \emph{Distributions {\'{a}} valeurs vectorielles}, Annales de
  l'Institut Fourier (tome 7, 1957, tome 8, 1959).

\bibitem[Sch66]{schwartz1}
\bysame, \emph{Th{\'{e}}orie des distributions}, Hermann, Paris, 1966.

\bibitem[Sch03]{schwartzman2}
Sol Schwartzman, \emph{Higher dimensional asymptotic cycles}, Canadian Journal
  of Mathematics \textbf{55} (2003), no.~3, 636--648.

\bibitem[Sul76]{sullivan}
Dennis Sullivan, \emph{Cycles for the dynamical study of foliated manifolds and
  complex manifolds}, Inventiones mathematicae \textbf{36} (1976), 225--255.

\bibitem[Tho01]{thomas}
Erik Thomas, \emph{Path integrals associated with sturm-liouville operators},
  J. Korean Math. Soc. \textbf{38} (2001), no.~2, 365--383.

\bibitem[Whi57]{whitney}
Hassler Whitney, \emph{Geometric integration theory}, Princeton University
  Press, Princeton, NJ, 1957.

\end{thebibliography}
\bibliographystyle{amsalpha}

\end{document}